\documentclass[10pt]{amsart}
\usepackage{amsmath}
\usepackage{amsthm}
\usepackage{amsfonts}
\usepackage{amssymb}
\usepackage[all]{xy}
\usepackage{amscd}

\usepackage{enumerate}

\usepackage{hyperref}
\hypersetup{
    colorlinks=true,      
    linkcolor=blue,          
    citecolor=blue,      
    filecolor=blue,      
    urlcolor=blue           
}

    \newtheorem{Lem}{Lemma}[section]
    \newtheorem{Lem-Def}{Lemma-Definition}[section]  
    \newtheorem{Prop}[Lem]{Proposition}
          \newtheorem*{Quest}{Question}   
    \newtheorem*{thm}{Theorem}   
    \newtheorem{Thm}[Lem]{Theorem}
    \newtheorem{Cor}[Lem]{Corollary}

\theoremstyle{definition}

    \newtheorem{Exa}[Lem]{Example}

\newcommand{\A}{\mathcal A}
\newcommand{\B}{\mathcal B}

\newcommand{\I}{\mathcal I}

\newcommand{\T}{\mathcal T}

\renewcommand{\S}{\mathcal S}
\renewcommand{\L}{\mathcal L}
\renewcommand{\P}{\mathcal P}
\renewcommand{\O}{\mathcal O}
\newcommand{\C}{\mathcal C}

\renewcommand{\:}{\colon}

\newcommand{\Diff}{\text{Diff}}

\newcommand{\ra}{\rightarrow}
\newcommand{\ol}{\overline}

\newcommand{\ze}{\mathbb{Z}}

\newcommand{\wed}{\wedge}
\newcommand{\wh}{\widehat}

\newcommand{\term}{\text{Term}}
\newcommand{\lra}{\longrightarrow}
\renewcommand{\:}{\colon}

\begin{document}

\title[The degree-2 Abel--Jacobi map for nodal curves - I]{The degree-2 Abel--Jacobi map for nodal curves - I}

\author{Marco Pacini}

\address{Marco Pacini, Universidade Federal Fluminense, Rua M. Braga, Niter\'oi (RJ) Brazil}

\email{pacini@impa.br, pacini@vm.uff.br}

 \thanks{The author was partially supported by CNPq, processo 300714/2010-6.}

\begin{abstract}
\noindent
Let $f\: \C\ra B$ be a regular local smoothing of a nodal curve.  In this paper, 
we find a modular description of the Abel--N\'eron map having values in Esteves's fine compactified Jacobian and extending the degree 2 Abel--Jacobi map of the generic fiber of $f$.  
\end{abstract}
 
\maketitle

\section{Introduction}

\noindent
Let $C$ be a smooth projective curve defined over an algebraically closed field $k$. Let $J_C$ be the Jacobian variety of $C$. For every positive integer $d$ and for every line bundle $\P$ of degree $d$ on $C$, the \emph{degree-$d$ Abel map of $C$} is the morphism $\alpha^d_\P\:  C^d\ra J_C$ associating to a $d$-tuple $(Q_1,\dots, Q_d)$ the isomorphism class of the degree-0 line bundle 
$\P\otimes\O_C(-\sum_{i=1}^d Q_i)$ on $C$. The \emph{degree-$d$ Abel--Jacobi map of $C$} is the morphism 
$\alpha^d_{\O_C(dP)}$, where $P$ is a point of $C$. 
A well-known result of Abel states that the fibers of the Abel map are projectivized complete linear series (up to the natural action of the $d$-th symmetric group). 

It is natural and useful to investigate how limit linear series degenerate when $C$ specializes to a singular curve. For example, the study of degenerations of limit linear series to singular curves 
  provided a proof of the celebrated Brill--Noether Theorem (see \cite{GH}). A systematic theory of limit linear series for curves of compact type was introduced by Eisenbud and Harris in \cite{EH}. 
  Significant progresses in describing limit canonical series were done by Esteves and Medeiros in \cite{EM} for nodal curves with two components.  
Recently, Osserman introduced in \cite{O} another construction for the basic theory of limit linear series  for curves of compact type. Nevertheless, a general theory of limit linear series for singular curves is still not available. 

The Abel's result suggests a possible new approach for the study of limit linear series on singular curves. The relationship between limit linear series and fibers of Abel maps has been explored in \cite{EO} for curves of compact type with two component. However, a systematic study of limit linear series through Abel maps for more complex types of curves should require the construction of degree-$d$ Abel maps for singular curves. We recall that degree-$d$ Abel maps have been constructed only in few cases: For integral curves in \cite{AK}; for stable curve and $d=1$ in \cite{CE};  for Gorenstein curves and $d=1$ in \cite{CCE}; for nodal curves with two components and two nodes and $d=2$ in \cite{Co}; for stable curves of compact type and any $d$ in \cite{CP}. 
The general problem is difficult and remains wide open, principally due to the combinatorial complexity of some issues, as we will explain further down.

To better understand the problem, we resort to families of curves.
More precisely, let $C$ be a nodal curve defined over an algebraically closed field $k$ and with irreducible components  $C_1,\dots, C_p$. 
 Let $f\: \C\ra B$ be a regular local smoothing of $C$, i.e. a family of curves where $\C$ is smooth and where $B$ be the spectrum of a Henselian DVR (discrete valuation ring) with residue field $k$ and  quotient field $K$, and such that $f$ has special fiber isomorphic to $C$ and smooth generic fiber $\C_K$. 
Let $\sigma\:  B\ra \C$ be a section of $f$ through the $B$-smooth locus of $\C$ such that $\sigma(Spec(k))$ is contained in $C_1$. Assume that $\mathcal E$ is a vector bundle on $\C$ of rank $r>0$ and degree $r(g-1)$, where $g$ is the genus of $C$. The vector bundle $\mathcal E$ is usually called \emph{a polarization on $\C/B$}. Consider  the compactified Jacobian $J^\sigma_{\mathcal E}$ constructed in \cite{E01} by Esteves, parametrizing degree-0 torsion-free rank 1 sheaves $\I$ on $\C/B$ such that 
$\I|_C$ is \emph{$C_1$-quasistable with respect to $\mathcal E$}. This means that $\I|_C$ satisfies certain numerical conditions depending on the dual graph of $C$. We recall that $J^\sigma_{\mathcal E}$ is a proper $B$-scheme. Let $\C^d_K$ be the product of $d$ copies of $\C_K$ over $B$ and, for every line bundle $\P$ of relative degree $d$ on $\C/B$, consider the degree-$d$ Abel map 
$$
\alpha^d_{\P,K}\: \C^d_K \longrightarrow  J^\sigma_{\mathcal E},
$$
  sending a $d$-tuple of points $(Q_{1,K},\dots, Q_{d,K})$ on $\C_K$ to $\P|_{\C_K}\otimes\O_{\C_K}(-\sum_{i=1}^d Q_{i,K})$.  It is worth to recall that other compactified Jacobians have been also employed as targets of Abel maps, for example the one constructed by Caporaso in \cite{C}.
  We can see $\alpha^d_{\P,K}$ as a rational map $\alpha^d_{\P,K}\: \C^d\dashrightarrow 
 J^\sigma_{\mathcal E}$, where $\C^d$ is the product of $d$ copies of $\C$ over $B$.
The problem of constructing a geometrically meaningful Abel map for $C$ turns into the problem of   describing a resolution of the rational map $\alpha^d_{\P,K}$  through a sequence of explicit blowups of $\C^d$.

Since Abel maps exist for any nodal curve only in degree 1 (at least for certain polarizations),  the natural next step is to consider the degree-2 case. In \cite{CEP}, the question whether or not it is possible to obtain a resolution of the rational map $\alpha^2_{\P, K}$ through a sequence of blowups along certain divisors of $\C^2$, is reduced to a series of combinatorial issues. In this paper and in \cite{P} we solve the posed combinatorial problems when $\mathcal E$ is the \emph{canonical polarization on $\C/B$} (see Section \ref{intro-sec}) and for $\P=\I_{\Sigma|\C}^{-2}$,
where $\I_{\Sigma|\C}$ is the ideal sheaf of $\Sigma:=\sigma(B)$,  i.e. for the degree-2 Abel--Jacobi map. The goal of the two papers is to prove that a resolution of the map $\alpha^2_{\I_{\Sigma|\C}^{-2},K}\:  \C^2\dashrightarrow J^\sigma_{\mathcal E}$ can be obtained by taking the blowup of $\C^2$ along products of subcurves of $C$ intersecting their complementary subcurves in 2 or 3 points.
 
A recent result of Busonero, Kass and Melo--Viviani (see \cite{B}, \cite{K}, \cite{MeVi}) shows that the N\'eron model of the Jacobian variety of the generic fiber of $f$ is isomorphic to the $B$-smooth locus of 
$J^\sigma_{\mathcal E}$. This is as an extension of a previous result of Caporaso on compactified Jacobians and N\'eron models (see \cite{CAJM}).  The N\'eron mapping property (see Section \ref{ext-sec}) implies a natural extension 
$$\alpha^d_{\P}\:  \dot{\C}^d\longrightarrow J^\sigma_{\mathcal E}$$
of the Abel map $\alpha^d_{\P,K}$, where $\dot{\C}^d$ is the $B$-smooth locus of $\C\times_B\C$. The morphism $\alpha^d_{\P}$ is known as \emph{Abel--N\'eron map}. Unfortunately, the definition of the Abel-N\'eron map $\alpha^d_{\P}$ is not explicit and its modular interpretation turns out to be a necessary step toward a geometrically meaningful resolution of the Abel map,
as we shall see in \cite{P}. Indeed, we recall that the degree-1 Abel map has been  constructed using the modularity of the Abel--N\'eron map in \cite[Theorem 4.6]{CE}. 
 To obtain a description of $\alpha^d_{\P}$ when $\mathcal E$ is the canonical polarization on $\C/B$ and $\P=\I_{\Sigma|\C}^{-2}$, we are naturally led to the following combinatorial question.

\begin{Quest}
Let $P$, $Q$ and $Q'$ be smooth points of $C$ contained respectively in $C_1$, $C_i$ and $C_j$, for some $(i, j)$ in $\{1,\dots,p\}^2$. Is it possible to find explicit integers $a_1,\dots,a_p$ such that the line bundle $$\O_C(2P-Q-Q')\otimes \O_\C\left(-\sum_{i=1}^p a_i C_i\right)|_C$$
on $C$ is $C_1$-quasistable?
\end{Quest}

The answer for the analogous question in degree 1 involves the construction of a set of nested  subcurves of $C$ intersecting their complementary in 1 points (see  \cite[Lemma 4.9]{CE}).
We answer the posed question by constructing  the \emph{set of nested tails of $C$ with respect to $(i,j)$}, consisting of certain subcurves of $C$ explicitly given in terms of $P,Q,Q'$. This set is easily computable, as Example \ref{Exa} clearly illustrates.

The main results of the paper are stated in the following theorem.
\begin{thm}
Let $C$ be a nodal curve defined over an algebraically closed field $k$, with irreducible components $C_1,\dots, C_p$. 
Let $f\: \C\ra B$ be a regular local smoothing of $C$, where $B$ is the spectrum of a Henselian DVR with residue field $k$.  Fix smooth points $P$, $Q$, and $Q'$ of $C$ contained respectively in  $C_1$, $C_i$ and $C_j$, where $(i,j)$ is in $\{1,\dots,p\}^2$. If $\T_{i,j}$ is the set of nested tails of $C$ with respect to $(i,j)$, then the  invertible sheaf
$$\O_C(2P-Q-Q')\otimes \O_\C\left(-\sum_{Z\in \T_{i,j}}Z\right)|_C$$
 on $C$ is $C_1$-quasistable. 
In particular, let $\sigma\:  B\ra \C$ be a section of $f$ through the $B$-smooth locus of $\C$ such that $\sigma(Spec(k))$ is contained in $C_1$ and set $\Sigma:=\sigma(B)$. If $\mathcal E$ is the canonical polarization on $\C/B$, then the Abel--N\'eron map $\alpha^2_{\I_{\Sigma|\C}^{-2}}\:  \dot{\C}^2\ra J^\sigma_{\mathcal E}$ is induced by the invertible sheaf $\dot\L$ on $\dot\C^2\times_B\C/\dot\C^2$ defined in  (\ref{Ldot}).
\end{thm}

\subsection{Notation and Terminology}\label{sec1.2}

We work over an algebraically closed field $k$.  
A \emph{curve} is a connected, projective and reduced scheme of dimension 1 over $k$. Let $C$ be a 
 nodal curve. The genus of $C$ is $g=1-\chi(\O_C)$. We denote by $\omega_C$ the dualizing sheaf of $C$. 
 We say that a subset $\Delta$ of the set of nodes of $C$ is a \emph{desconnecting subset} if the normalization of $C$ at the points of $\Delta$ is not connected. We say that a node $R$ of $C$ is a \emph{desconnecting node} if $\{R\}$ is a desconnecting subset. 
A \emph{subcurve} $Z$ of $C$ is a nonempty union of irreducible components of $C$ such that $Z\ne C$. If $Z$ is a subcurve of $C$, then its \emph{complementary subcurve} is $Z^c:=\ol{C\setminus Z}$. We call a point in $Z\wedge Z^c$ a \emph{terminal point of $Z$}, and we  set   $\term_Z:=Z\cap Z^c$ and $k_Z:=\#\term_Z$. Moreover, we set $\term_C=\term_\emptyset=\emptyset$. 

Let $Z$ and $Z'$ be subcurves of a nodal curve $C$.  
 We write $Z\lhd Z'$ if $Z\subsetneq Z'$ and $\term_Z\cap\term_{Z'}$ is empty. Moreover, we write $Z\wed Z'$ to denote the union of the irreducible components of $C$ contained in $Z\cap Z'$. Notice that 
$$(Z\wed Z')^c=Z^c\cup (Z')^c.$$
 If  $\term_Z\cap \term_{Z'}$ is nonempty, we say that the pair $(Z, Z')$ is \emph{terminal}, or that $Z$ is $Z'$-terminal, or that $Z'$ is $Z$-terminal.  Otherwise, we say that $(Z, Z')$ is \emph{free}.  If $\S$ is a set of subcurve of $C$, we say that $Z$ is $\S$-free if $(Z, W)$ is free, for every $W$ in $\S$. 
 We say that $(Z, Z')$ is \emph{perfect} if one of the following condition holds 
 $$Z\subseteq Z', \,\,\, Z'\subseteq Z, \,\,\, Z^c\subseteq Z', \,\,\,Z'\subseteq Z^c.$$
If $\S$ is a set of subcurves of $C$, we say that $Z$ is \emph{$\S$-normalized} if $(Z, W)$ is perfect, for every $Z$-terminal subcurve $W$ in $\S$. 
For every node $R$ of $C$, we let  $C_{R,1},C_{R,2}$ be the irreducible components of $C$ containing $R$.   If $\A$ and $\B$ are sets, we let $\Diff(\A,\B):=(\A\cup \B)\setminus(\A\cap \B)$; we denote by $\A\sqcup \B$ the disjoint union of $\A$ and $\B$.
If $I$ is a torsion-free rank 1  sheaf on $C$, its degree is $\deg I=\chi(I)-\chi(\O_C)$.

A \emph{family of nodal curves} is a proper and flat morphism 
$f\: \mathcal C\rightarrow B$ whose geometric fibers are nodal 
curves. We denote by $\omega_f$ the relative dualizing sheaf of the family. 
A \emph{local smoothing} of a nodal curve $C$ is a family of curves $f\: \C\ra B$, where $B$ is the spectrum of a Henselian DVR with residue field $k$ and quotient field $K$ and such that $f$ has special fiber isomorphic to $C$ and smooth generic fiber. 
A \emph{regular local smoothing} $f\: \C\ra B$ of $C$ is a local smoothing of $C$ with 
$\C$ smooth.

\section{Jacobians and Abel--Jacobi maps for nodal curves}\label{intro-sec}

\noindent
Let $C$ be a genus $g$ nodal curve with irreducible components $C_1,\dots, C_p$. Let $J_C$ be the Jacobian of $C$, a scheme parametrizing invertible sheaves of degree 0 on $C$. We have a natural decomposition 
$$J_C=\underset{d_1+\ldots+d_p=0}{\underset{(d_1,\dots,d_p)\in\ze^p}
{\coprod}}J^{(d_1,\dots,d_p)}_C,$$
where $J^{(d_1,\dots,d_p)}_C$ is a connected component of $J_C$ parametrizing invertible sheaves $I$ on $C$ such that $\deg_{C_i}I=d_i$, for $i$ in $\{1,\dots, p\}$.
In general, the scheme $J_C$ is neither of finite type, nor compact. To consider a manageable compactification of it, we resort to a semistability condition and to torsion-free rank 1 sheaves. 
Consider the vector bundle
$$E:=
\begin{cases}
\begin{array}{ll}
\O_C^{\oplus(2g-3)}\oplus \omega_X^{\otimes g-1} & \text{ if } g\ge 2 \\
\O_C & \text{ if } g=1 \\
\O_C\oplus \omega_C  & \text{ if } g=0 \\
\end{array}
\end{cases}$$
on $C$. The vector bundle $E$ is called the \emph{canonical polarization on $C$}.
Let $I$ be a degree-0 torsion-free rank 1 sheaf on $C$. 
 For every subcurve $Z$ of $C$, we say that $I$ is \emph{semistable with respect to $E$ at $Z$}, or simply \emph{(canonically) semistable at $Z$},  if  
$$
|\deg_Z I_Z|\le \frac{k_Z}{2},
$$
where $I_Z$ is the restiction of $I$ to $Z$ modulo torsion. 
Furthermore, for every subcurve $Z$ of $C$ and every component $C_i$ of $C$, we say that $I$ is \emph{$C_i$-quasistable with respect to $E$ at $Z$}, or simply \emph{(canonically) $C_i$-quasistable at $Z$}, if $I$ is semistable at $Z$ 
and, whenener $C_i\subseteq Z$, we have
$$\beta_{I}(Z):=\deg_Z I_Z+\frac{k_Z}{2}>0.$$
Notice that  $\beta_I(Z)\in\ze+\frac{1}{2}\ze$, for every subcurve $Z$ of $C$.
We say that $I$ is \emph{$C_i$-quasistable with respect to $E$}, or simply \emph{(canonically) $C_i$-quasistable}, if $I$ is $C_i$-quasistable at $Z$, for every subcurve $Z$ of $C$.    
 It follows from \cite[Theorem A]{E01} that there exists a scheme $J^{C_i}_C$, which is of finite type and proper, parametrizing the set of $C_i$-quasistable torsion-free rank 1 sheaves on $C$.
  We refer to \cite{E01} for more details (see also \cite[Section 2.3]{CCE}).   
Notice that, assuming that $I$ is invertible, it follows that $I$ is semistable at a subcurve $Z$ of $C$ if and only if $I$ is semistable at $Z^c$. Moreover, if $I$ is invertible and $C_i$-quasistable at the connected components of a subcurve $Z$ of $C$, then $I$  is $C_i$-quasistable at $Z$.  

The definitions extend to families of curves in a natural way. 
Let $f\: \C\ra B$ be a family of nodal curves.  Assume that there are sections $\sigma_1,\dots,\sigma_n\:  B\ra \C$ through the $B$-smooth locus of $\C$ such that, for every $b\in B$ and for every  irreducible component $X_b$ of $f^{-1}(b)$, we have $\sigma_i(b)\in X_b$, for some $i\in\{1,\dots, n\}$.  Notice that this condition is satisfied if $f$ is a regular local smoothing of a nodal curve (see \cite[Proposition 5 of Section 2.3]{BLR}). Let $\sigma\:  B\ra \C$ be a section of $f$ through the $B$-smooth locus of $\C$.  
 Consider the vector bundle 
$$\mathcal E:=
\begin{cases}
\begin{array}{ll}
 \O_\C^{\oplus(2g-3)}\oplus \omega_f^{\otimes g-1} & \text{ if } g\ge 2 \\
\O_\C & \text{ if } g=1 \\
\O_\C\oplus \omega_f  & \text{ if } g=0
\end{array}
\end{cases}$$
on $\C/B$. 
The vector bundle $\mathcal E$ is called the \emph{canonical polarization on $\C/B$}.
 We say that a  torsion-free rank 1 sheaf $\I$ on $\C$ is \emph{$\sigma$-quasistable with respect to $\mathcal E$}, or simply \emph{(canonically) $\sigma$-quasistable}, if $\I|_{f^{-1}(b)}$ is $X_b$-quasistable,  for every $b\in B$, where $X_b$ is the irreducible component of $f^{-1}(b)$ such that 
 $\sigma(b)\in X_b$. It follows from \cite[Theorems A and B]{E01} that there exists a scheme $J^\sigma_{\mathcal E}$ which finely represents the functor associating to a $B$-scheme $T$ the set of equivalence classes of $\sigma_T$-quasistable torsion-free rank 1 sheaves on $\C\times_B T/T$, where $\sigma_T\:  T\ra \C\times_B T$ is the pull-back of $\sigma$. Here, two torsion-free rank 1 sheaves $\I_1$ and $\I_2$ on $\C\times_B T/T$ are equivalent if there is an invertible sheaf $M$ on $T$ such that $\I_1\simeq \I_2\otimes p^*M$, where $p\:  \C\times_B T\ra T$ is the second projection. The scheme $J^\sigma_{\mathcal E}$ is of finite type and proper over $B$. 

Let $f\:  \C\ra B$ be a regular local smoothing of a nodal curve $C$, where $B$ is the spectrum of a Henselian DVR with quotient field $K$. Let $\mathcal E$ be the canonical polarization on $\C/B$. Let $\sigma\:  B\ra \C$ be a section through the $B$-smooth locus of $\C$.  
 Set  $\C^2:=\C\times_B\C$ and $\C^3:=\C^2\times_B\C$.   
 Denote by $\xi\:  \C^3\ra \C$ and $\rho_i\: \C^3\ra\C^2$ the projection onto the last factor and that onto the product over $B$ of the $i$-th and last factor, for each $i$ in $\{1,2\}$.  Let $\Delta\subset\C^2$ be the diagonal subscheme and, for each $i$ in $\{1,2\}$, put
 $$\Delta_i:=\rho_i^{-1}(\Delta)$$
and consider the ideal sheaf $\I_{\Delta_i|\C^3}$. Consider the ideal sheaf $\I_{\Sigma|\C}$, where $\Sigma:=\sigma(B)$. 
The degree 2 Abel--Jacobi map of the generic fiber $\C_K$ of $f$ is the morphism 
\begin{equation}\label{AJ-generic}
\alpha^2_{\I_{\Sigma|\C}^{-2},K}\:  \C_K\times_B \C_K\ra J^\sigma_{\mathcal E}
\end{equation}
induced by the  invertible sheaf $$(\xi^*\I_{\Sigma|\C}^{-2}\otimes\I_{\Delta_1|\C^3}\otimes\I_{\Delta_2|\C^3})|_{\C_K\times_B \C_K\times_B\C}$$
on the family $\rho\: \C_K\times_B \C_K\times_B\C\ra\C_K\times_B \C_K$, where $\rho$ is the projection onto the first and second factor.

\section{Tails of nodal curves}\label{sec3}

\noindent
In the literature, a tail of a nodal curve is a subcurve intersecting its complementary curve exactly at one point (see for example \cite[Definition 4.1]{CE}).  We need to generalize the notion of tail of a nodal curve as follows.

Let $Z$ be a subcurve of a nodal curve.  We say that $Z$ is a \emph{tail} if $Z$ and $Z^c$ are  connected.  For a positive integer $k$,  a \emph{k-tail} is a tail $Z$ such that $k_Z=k$.

\begin{Lem}\label{term-proof}
Let $Z$ and $Z'$ be subcurves of a nodal curve $C$. Then we have 
\begin{equation}\label{term}
\term_{Z\wed Z'}\cup \term_{Z\cup Z'}\subseteq\term_Z\cup \term_{Z'}.
\end{equation}
If $(Z, Z')$ is free, then $\term_{Z\wed Z'}\cap \term_{Z\cup Z'}$ is empty and the equality holds in~\eqref{term}.
\end{Lem}

\begin{proof}
 Let $R$ be a terminal point of $Z\wed Z'$, with $C_{R,1}$ contained in $Z\wed Z'$ and $C_{R,2}$  in $(Z\wed Z')^c=Z^c\cup (Z')^c$.  If $C_{R,2}$ is contained in $Z^c$ (respectively in $(Z')^c)$, then $R$ is a terminal point of $Z$ (respectively of $Z'$). Similarly, any terminal point of $Z\cup Z'$ is a terminal point of either 
 $Z$ or $Z'$. The proof of~\eqref{term} is complete.

Suppose that $(Z, Z')$ is free and, by contradiction, that there is a terminal point $R$ of both 
$Z\wed Z'$ and $Z\cup Z'$, with 
$C_{R,1}$ contained in $Z\wed Z'$. Since $R$ is in $\term_{Z\cup Z'}$ and $C_{R,1}\subseteq Z\cup Z'$, it follows that $C_{R,2}\subseteq (Z\cup Z')^c=Z^c\wedge (Z')^c$, and hence $R$ is a terminal point of both $Z$ and $Z'$, which contradicts the fact that $(Z, Z')$ is free.

Suppose $(Z, Z')$ free and $R$ a terminal point of $Z$, with $C_{R,1}$ contained in $Z$ and 
$C_{R,2}$ in $Z^c$. Since $(Z, Z')$ is free, we have two possibilities: either $C_{R,1}\cup C_{R,2}\subseteq Z'$, and hence $R$ is a terminal point of $Z\wed Z'$,  or 
 $C_{R,1}\cup C_{R,2}\subseteq (Z')^c$, and hence $R$ is a terminal point of $Z\cup Z'$.  Similarly, any terminal point of $Z'$ is a terminal point of either $Z\wedge Z'$ or $Z\cup Z'$, and hence the other inclusion in~\eqref{term} holds.
\end{proof}

\begin{Lem}\label{term-lem}
Let $Z$ be a tail and $W, W'$ be subcurves of a nodal curve $C$ and set $\A:=\term_W$, $\B:=\term_{W'}$. 
The following properties hold
\begin{enumerate} [(i)]
\item  \label{term-lem(i)}
if $W^c$ is contained in $W'$, then $\term_{W\wed W'}$ is equal to $\Diff (\A,\B)$;
\item  \label{term-lem(ii)}
if  $W'$ is contained in $W^c$, then $\term_{W\cup W'}$ is equal to $\Diff (\A,\B)$;
\item  \label{term-lem(iii)}
if  $Z$ is contained in $W\wed W'$, then $\term_Z\cap (\A\cup \B)$ is contained $\term_{W\wed W'}$;
\item  \label{term-lem(iv)}
if  $W\cup W'$ is contained in $Z$, then $\term_Z\cap (\A\cup \B)$ is contained in $\term_{W\cup W'}$;
\item  \label{term-lem(v)}
if $W$ is contained in $Z$ and $W'$ in $Z^c$, then $\term_{W\cup W'}$ is equal to $\Diff(\A,\B)$.
\end{enumerate}
\end{Lem}

\begin{proof}
Set $X:=W\wed W'$ and $X':=W\cup W'$.
The statement is clear if $X$ is empty or $X'=C$, thus we may assume $X$ nonempty and 
$X'$ different from $C$. The items (ii) and (iv) follow by items (i) and (iii), by taking complementary subcurves.

We show (i). Suppose $W^c\subseteq W'$. Let 
$R$ be a terminal point of $X$, with $C_{R,1}$ contained in $X$ and $C_{R,2}$ in $X^c=W^c\cup (W')^c$. It follows from (\ref{term}) that $R$ is in $\A\cup \B$. Moreover, either $C_{R,2}\subseteq W^c\subseteq W'$, and hence $R$ is not in $\B$, or $C_{R,2}\subseteq (W')^c\subseteq W$, and hence $R$ is not in $\A$. Conversely, let $R$ be in $\A\setminus\B$, with $C_{R,1}$ contained in $W$ and $C_{R,2}$  in $W^c$. Notice that $C_{R,2}$ is contained in $W'\wed X^c$. Since $R$ is not in $\B$,  it follows that $C_{R,1}\subseteq W'$, and hence $C_{R,1}\subseteq W\wedge W'=X$. Since $C_{R,1}$ is contained in $X$ and $C_{R,2}$ in $X^c$,  we see that $R$ is a terminal point of $X$.  Similarly, we have that $\B\setminus\A$ is contained in $\term_X$. 

We show (iii). Suppose $Z\subseteq X$. Let $R$ be a terminal point of  $Z$ and $W$, with $C_{R,1}$ contained in $Z$ and $C_{R,2}$ in $Z^c$. 
Notice that $C_{R,2}\subseteq X^c$, otherwise $C_{R,2}\subseteq W$, and hence $R$ would be  not a terminal point of $W$, a contradiction. Since $C_{R,1}$ is contained in $X$ and $C_{R,2}$ in $X^c$, it follows that $R$ is a terminal point of $X$. One can show similarly that the intersection of 
  $\term_Z$ and $\B$ is contained in $\term_X$. 

We show (v). Suppose $W\subseteq Z$ and $W'\subseteq Z^c$, and consider a terminal point 
$R$ of $X'$, with $C_{R,1}$ contained in $X'$ and $C_{R,2}$ in $(X')^c=W^c\wed (W')^c$.  It follows from (\ref{term}) that $R$ is in $\A\cup \B$. Nevertheless, the node $R$ is not in $\A\cap \B$, otherwise $C_{R,1}$ would be contained in $W\wed W'\subseteq Z\wed Z^c$, which is empty,   a contradiction.  Conversely,  let
 $R$ be in  $\A\setminus \B$, with $C_{R,1}$ contained in $W$ and $C_{R,2}$ in $W^c$.  Since $W\wedge W'$ is empty, we have $C_{R,1}\subseteq (W')^c$; since $R$ is not in $\B$, we see that $C_{R,2}\subseteq (W')^c$. Therefore,  
 $C_{R,1}\subseteq  W$ and $C_{R,2}\subseteq W^c\wed (W')^c$, and we conclude that $R$ is a terminal point of $X'$. One can show similarly that $\B\setminus \A$ is contained in $X'$. 
\end{proof}

\begin{Lem}\label{connect}
 Let $Z$ and $Z'$ be  tails  of a nodal curve $C$ such that $k_Z>1$ and $k_{Z'}>1$.  Then the following properties hold
 \begin{enumerate} [(i)]
\item  \label{connect(i)}
 if  $k_{Z\wed Z'}$ is in $\{1,2,3\}$, then $k_{Z\wedge Z'}>1$ and $Z\wed Z'$ is a tail;
 \item \label{connect(ii)}
 if $k_{Z\cup Z'}$ is in $\{1,2,3\}$, then $k_{Z\cup Z'}>1$ and $Z\cup Z'$ is a tail;
 \item \label{connect(iii)}
  if $(Z, Z')$ is free, $k_{Z\wed Z'}\ge1$, $k_{Z\cup Z'}\ge1$ and $(k_Z,k_{Z'})=(2, 3)$, then $k_{Z\wed Z'}>1$, $k_{Z\cup Z'}>1$, $k_{Z\wed Z'}+k_{Z\cup Z'}=5$ and $Z\wed Z'$ and $Z\cup Z'$ are tails.
  \end{enumerate}
\end{Lem}

\begin{proof}
We prove (i). 
Suppose $k_{Z\wed Z'}$ is in $\{1,2,3\}$. 
Notice that  $k_{Z\wed Z'}>1$,  otherwise, using (\ref{term}), a terminal point of either 
$Z$ or $Z'$ would be a desconnecting node, which is not possible because $Z$ and $Z'$ are tails. By contradiction, assume that $Z\wed Z'$ is not a tail. Then there is  a partition of $\term_{Z\wed Z'}$  into nonempty proper subsets $U$ and $V$ which are  desconnecting sets of nodes of $C$.  
Since $k_{Z\wed Z'}\le 3$, one between $U$ and $V$ has cardinality one.
We know by (\ref{term}) that $U\cup V$ is contained in $\term _Z\cup \term_{Z'}$, hence either $Z$ or $Z'$ has a desconnecting node as terminal point, and this is a contradiction. The proof of (ii) is similar. 

  Suppose now that the hypotheses of (iii) hold. 
  Arguing as in the proof of (i), we have $k_{Z\wed Z'}>1$ and $k_{Z\cup Z'}>1$. Moreover, it follows from  Lemma \ref{term-proof} that the intersection of $\term_{Z\wed Z'}$ and $\term_{Z\cup Z'}$ is empty and 
   $$\#(\term_{Z\wed Z'}\cup\term_{Z\cup Z'})=\#(\term_Z\cup \term_{Z'})=5,$$
   where the second equality holds because $(Z, Z')$ is free. 
 We get $k_{Z\wed Z'}+k_{Z\cup Z'}=5$, then 
$k_{Z\wed Z'}$ and $k_{Z\cup Z'}$ are in $\{2,3\}$; moreover, $Z\wed Z'$ and $Z\cup Z'$ are tails by (i).
\end{proof}

\begin{Lem}\label{free-perf}
Let $Z$ be a tail of a nodal curve $C$. The following properties hold
 \begin{enumerate} [(i)]
\item  \label{free-perf(i)}
if $\term_Z\subset Z'$, for some tail $Z'$ of $C$, then either $Z\subseteq Z'$, or $Z^c\subseteq Z'$;
 \item \label{free-perf(ii)}
if $\#(\term_Z\cap\term_{Z'})=k_Z-1$,  for some tail $Z'$ of $C$, then $(Z,Z')$ is perfect;
 \item \label{free-perf(iii)}
 if $k_Z\ge 2$, then $(Z, Z')$ is free, for every 1-tail $Z'$ of $C$.
\end{enumerate}
\end{Lem}

\begin{proof}
We prove (i). Suppose $\term_Z\subset Z'$, where $Z'$ is a tail. Write $(Z')^c=W_1\cup W_2$, where $W_1:=(Z')^c\wed Z$ and $W_2:=(Z')^c\wed Z^c$. Notice that $W_1$ and $W_2$ are subcurves of $(Z')^c$ with no common components, hence  $W_1\cap W_2$ is contained in $(Z')^c\setminus Z'$. On the other hand, we have $W_1\cap W_2\subseteq Z\cap Z^c=\term_Z\subset Z'$. We conclude that $W_1\cap W_2$ is empty.    By contradiction, assume that $Z$ is not contained in $Z'$ and that $Z^c$ is not contained in $Z'$. It follows that $W_1$ and $W_2$ are nonempty, and hence $(Z')^c$ is not connected, and this is a contradiction because $Z'$ is a tail.

We now prove (ii) and (iii). If $\#(\term_Z\cap\term_{Z'})=k_Z-1$, then either $\term_Z\subset Z'$, or $\term_Z\subset (Z')^c$, and hence $(Z, Z')$ is perfect by the first part of the proof. 
If $k_Z\ge 2$, then any proper subset of $\term_Z$ is not a desconnecting set, because $Z$ is a tail, and hence  $(Z, Z')$ is free, for every 1-tail $Z'$ of $C$.
\end{proof}

\section{Sets of nested tails}\label{nest-sec}

\noindent
Let $f\: \C\ra B$ be a regular local smoothing of a nodal curve $C$ with irreducible components $C_1,\dots, C_p$.  Let $P$, $Q$, and $Q'$ be smooth points of $C$, with $P$ in $C_1$. We want to find explicit integers $a_1,\dots, a_p$ such that the invertible sheaf 
$$\O_C(2P-Q-Q')\otimes \O_\C\left(-\sum_{i=1}^pa_iC_i\right)$$
on $C$ is $C_1$-quasistable. Recall that a similar result for the line bundle $\O_C(Q)$, where $Q$ is a smooth point of $C$, has been obtained via sets of 1-tails of $C$ in \cite[Lemma 4.9]{CE}. To determine the integers in terms of the points $P$, $Q$, and $Q'$, we need to introduce certain sets of $k$-tails of $C$, for $k$ in $\{1,2,3\}$.

 Let $C$ be a nodal curve with irreducible components $C_1,\dots,C_p$. Fix positive integers $r$ and $s$, and an $r$-tuple $(i_1,\dots,i_r)$ in $\{1,\dots,p\}^r$. We say that a set $\T$ is a \emph{set of nested $s$-tails of $C$ with respect to $(i_1, \dots, i_r)$} if 
$$\T=\{W_0,\dots,W_m\},$$
where $m$ is a non-negative integer and $W_0,\dots,W_m$ are $s$-tails of $C$ satisfying the following conditions
\begin{enumerate}
\item 
if $t$ is in $\{0,\dots, m\}$, then $\cup_{u=1}^r C_{i_u}$ and $C_1$ are contained respectively in $W_t$ and $W_t^c$;
\item 
if $m\ge 1$ and  $t$ is in $\{0,\dots,m-1\}$, then we have $W_t\lhd W_{t+1}$.
\end{enumerate}

Fix $i$ in $\{1,\dots,p\}$. Consider the following set of 1-tails
$$\T^1_i:=\{Z: k_Z=1, C_i\subseteq Z \text{ and }C_1\subseteq Z^c\}.$$
By \cite[Lemma 4.3]{CE}, $\T^1_i$ is a sets of nested 1-tails of $C$ with respect to $(i)$. 

\begin{Prop}\label{2tail}
Let $Z$ and $Z'$ be 2-tails of a nodal curve $C$ such that $k_{Z\wed Z'}\ge 1$ and $k_{Z\cup Z'}\ge1$. Then $Z\wed Z'$ and $Z\cup Z'$ are 2-tails of $C$.
\end{Prop}

\begin{proof}
Suppose $\term_Z\subset Z'$. It follows from item (i) of Lemma \ref{free-perf} that either
 $Z\subseteq Z'$, or $Z^c\subseteq Z'$. If $Z\subseteq Z'$, then $Z\wedge Z'=Z$ and $Z\cup Z'=Z'$, and  we are done. If $Z^c\subseteq Z'$, then $Z\cup Z'=C$, a contradiction. We can argue similarly if one of the following conditions holds: $\term_{Z'}\subset Z$, $\term_Z\subset (Z')^c$, $\term_{Z'}\subset Z^c$.

By the first part of the proof, we may assume that $\term_Z$ is equal to $\{R,S\}$, with $R$ and $S$ not contained respectively in $Z'$ and $(Z')^c$, and that $\term_{Z'}$ is equal to $\{R',S'\}$, with $R'$ and $S'$ not contained respectively in $Z$ and $Z^c$. 
As a consequence, the intersection $\{R,R'\}\cap (Z\wed Z')$ is empty, and  
$C_{S,1}\cup C_{S,2}$ and $C_{S',1}\cup C_{S',2}$ are contained respectively in $Z'$ and $Z$,  hence we deduce that 
$$\{R,R'\}\cap\term_{Z\wed Z'}=\{S,S'\}\cap \term_{Z\cup Z'}=\emptyset.$$ 
It follows from (\ref{term}) that $\term_{Z\wed Z'}$ is contained in $\{S,S'\}$ and $\term_{Z\cup Z'}$ in $\{R,R'\}$, and items (i) and (ii) of Lemma \ref{connect} implies that $Z\wedge Z'$ and $Z\cup Z'$ are 2 tails.
   \end{proof}

 Fix $(i,j)$ in $\{1,\dots,p\}^2$. Consider the following set of 2-tails 
$$\S^2_{i,j}:=\{Z: Z \text{ is a 2-tail of $C$ such that }  C_i\cup C_j\subseteq Z \text{ and }C_1\subseteq Z^c\}.$$

 If $\S^2_{i,j}$ is nonempty, set $W^2_0:=\wedge_{Z\in \S^2_{i,j}} Z$.  
It follows from Proposition \ref{2tail} that $W^2_0$ is in $\S^2_{i,j}$.
For every positive integer $m$, define inductively 
$$
 \S^2_{i,j,m}:=\{Z : Z\in \S^2_{i,j} \text{ and } W^2_{m-1}\lhd Z\}
$$
and, if $\S^2_{i,j,m}$ is nonempty, we let
$$W^2_m:=\wed_{Z\in \S^2_{i,j,m}} Z.
$$
It follows from (\ref{term}) that if $Z,Z',Z''$ are subcurve of $C$ such that $Z\lhd Z'$ and $Z\lhd Z''$, then $Z\lhd Z'\wed Z''$. Hence, Proposition \ref{2tail} implies that 
$W^2_m$ is in $\S^2_{i,j,m}$, for every positive integer $m$ such that $\S^2_{i,j,m}$ is nonempty. 
Let $M$ be the maximum positive integer such that $\S^2_{i,j,M}$ is nonempty and consider the set
$$\T^2_{i,j}:=\{W^2_0,\dots, W^2_{M}\}.$$
We call $\T^2_{i,j}$ \emph{the set of nested 2-tail of $C$ with respect to $(i, j)$}.

\begin{Cor}\label{maxT2}
Let  $C$ be a nodal curve with irreducible components $C_1,\dots C_p$. Let $\T^2_{i,j}$ be the set of nested 2-tails of $C$ with respect to $(i, j)$, where $(i,j)$ is in $\{1,\dots,p\}^2$.  If $Z$ is a 2-tail of $C$ such that $C_i\cup C_j$ and $C_1$ are contained respectively in  $Z$ and $Z^c$, then there is a $Z$-terminal tail in $\T^2_{i,j}$ which is contained in $Z$.
\end{Cor}

\begin{proof}
By the definition of $\T^2_{i,j}$, the tail $W^2_0$ of $\T^2_{i,j}$  is contained in $Z$. The result simply follows by observing that, if $W^2_m$ is  the maximal tail of $\T^2_{i,j}$ contained in $Z$ for $m\ge0$, then $W^2_m$ is $Z$-terminal. Indeed, if the pair $(W^2_m, Z)$ were free, then  $W^2_m\lhd Z$, and hence $W^2_m\lhd W^2_{m+1}\subseteq Z$, contradicting the maximality  of $W^2_m$.
\end{proof}

In the sequel, we will define sets of nested 3-tails of $C$. The following example shows that  it is necessary to introduce an additional condition to get the result stated in Proposition \ref{2tail} for 3-tails.

\begin{Exa}
It is possible that $k_{Z\wedge Z'}=2$ and $k_{Z\cup Z'}=4$, for 3-tails $Z$ and $Z'$ of a nodal curve $C$. For example, let $C=C_1\cup C_2\cup C_3\cup C_4$, where $\#C_1\cap C_4=\#C_2\cap C_3=0$, and $\#C_1\cap C_i=2$ and $\#C_i\cap C_4=1$, for $i\in\{2,3\}$. Then $Z=C_2\cup C_4$ and $Z'=C_3\cup C_4$ are 3-tails of $C$ such that $k_{Z\wedge Z'}=k_{C_4}=2$ and $k_{Z\cup Z'}=k_{C_1}=4$. 
Notice that $\T^2_{4,4}=\{C_4\}$. Indeed, the crucial fact here is that  $Z$ and $Z'$ are the 3-tails of $C$ containing $C_4$ and not containing $C_1$, and both tails are $\T^2_{4,4}$-terminal.
\end{Exa}

\begin{Lem}\label{3tail}
Let $Z$ and $Z'$ be 3-tails of a nodal curve $C$ such that $k_{Z\wed Z'}\ge 1$ and $k_{Z\cup Z'}\ge 1$. Then the following properties hold
 \begin{enumerate} [(i)]
\item  \label{3tail(i)}
if $\#(\term_{Z'}\cap Z)=2$, then $k_{Z\cup Z'}$ is in $\{2,3\}$.
\item  \label{3tail(ii)}
if $\#(\term_{Z'}\cap Z)=\#(\term_{Z}\cap Z')=1$, then $k_{Z\wed Z'}=2$.
\end{enumerate}
\end{Lem}

\begin{proof}
Write $\term_Z=\{R,S,T\}$ and $\term_{Z'}=\{R',S',T'\}$.

We show (i). 
Suppose  $\#(\term_{Z'}\cap Z)=2$, with $\{R',S'\}$ contained in $Z$ and $T'$ not contained in $Z$. Let us  prove that  $k_{Z\cup Z'}$ is in $\{2,3\}$. We distinguish two cases. 
Assume $\term_Z\cap \term_{Z'}=2$. It follows from item (ii) Lemma \ref{free-perf} that $(Z, Z')$ is perfect,  hence one of the following conditions holds 
$$Z\subseteq Z', \,\,\, Z'\subseteq Z, \,\,\, Z^c\subseteq Z', \,\,\, Z'\subseteq Z^c.$$  We are done if the first  or the second condition holds. If $Z^c\subseteq Z'$, then $Z\cup Z'=C$, while if $Z'\subseteq Z^c$, then $Z\wed Z'$ is empty, and  we get a contradiction in both cases. 
Assume $\term_Z\cap \term_{Z'}\ne 2$, with $R'$ in $Z$ and $S'$ not in  $Z^c$. We have $C_{S',1}\cup C_{S',2}\subseteq Z$, and hence $S'$ is not a terminal point of $Z\cup Z'$. Moreover, $Z'$ is connected, because it is a tail. Since $S'$ is not in $Z^c$ and $T'$ is not in $Z$, there is a terminal point $U$ of $Z$ such that $C_{U,1}\cup C_{U,2}\subseteq Z'$. As a consequence, we see that  $U$ is not a terminal point of $Z\cup Z'$, and 
it follows from (\ref{term}) that $\term_{Z\cup Z'}$ is contained in $\{R,S,T,R',T'\}\setminus\{U\}$, for $U$ in $\{R,S,T\}$. Since $Z$ and $Z'$ are 3-tails, the set $\term_{Z}\cup \term_{Z'}$ contains no separating nodes. Thus, using again (\ref{term}), we have $k_{Z\cup Z'}>1$. To conclude the proof of (i), we need only show that $k_{Z\cup Z'}\le 3$.  If $R'$ is not a terminal point of $Z$, then $R'$ is not in $Z^c$ and hence $C_{R',1}\cup C_{R',2}$ is contained in $Z$. In this case, $R'$ is not a terminal point of $Z\cup Z'$,  hence $k_{Z\cup Z'}\le 3$. On the other hand, if $R'$ is a terminal point of $Z$, then either $R'=U$ and it is not a terminal point of $Z\cup Z'$, or $R'$ is in $\{R,S,T\}\setminus\{U\}$. In any case, the inclusion 
$\term_{Z\cup Z'}\subseteq \{R,S,T,T'\}\setminus\{U\}$ holds, hence $k_{Z\cup Z'}\le 3$.

Suppose $\#(\term_{Z'}\cap Z)=\#(\term_{Z}\cap Z')=1$, where $\{R',S'\}\cap Z$ and $\{R,S\}\cap Z'$ are empty sets. Then $\{R,S,R',S'\}$ intersects $(Z\wed Z')$, and hence $\term_{Z\wed Z'}$, in the empty set. Therefore, from (\ref{term}), we see that the inclusion $\term_{Z\wed Z'}\subseteq \{T,T'\}$ holds. Arguing as for $Z\cup Z'$, we have $k_{Z\wed Z'}>1$, hence $k_{Z\wed Z'}=2$.
  \end{proof}

Let $C$ be a nodal curve with irreducible components $C_1,\dots, C_p$. 
For every $(i,j)$ in $\{1,\dots,p\}^2$, 
consider the set
$$\S^3_{i,j}:=\{Z: Z \text{ is a $\T^2_{i,j}$-free 3-tail of $C$ such that } C_i\cup C_j\subseteq Z \text{ and } C_1\subseteq Z^c\}.$$

\begin{Prop}\label{S3}
Let $C$ be  a nodal curve with irreducible components $C_1,\dots, C_p$ and let $(i,j)$ be in $\{1,\dots,p\}^2$. If $Z,Z'$ are in $\S^3_{i,j}$, then $Z\wed Z'$ is in $\S^3_{i,j}$.
\end{Prop}

\begin{proof}
 Of course, $C_i\cup C_j$ and $C_1$ are contained respectively in $Z\wed Z'$ and $(Z\wedge Z')^c$. 
  We claim that $Z\wed Z'$ is $\T^2_{i,j}$-free. Indeed, for every $W$ in $\T^2_{i,j}$, using (\ref{term}) we have
  \begin{equation}\label{termcup}
  \term_W\cap \term_{Z\wed Z'}\subseteq \term_W\cap (\term_Z\cup \term_{Z'}).
  \end{equation}
Since $Z$ and $Z'$ are $\T^2_{i,j}$-free, the right hand side of (\ref{termcup}) must be empty, and the claim follows. Therefore, to conclude the proof,  we only need to show that $Z\wed Z'$ is a 3-tail;  since $k_{Z\wed Z'}\ge 1$, using item (i) of Lemma \ref{connect}, we see that it suffices that $k_{Z\wed Z'}=3$. We distinguish three cases.

 Suppose $\#(\term_{(Z')^c}\cap  Z^c)=3$, i.e. $\term_{(Z')^c}$ is contained in $Z^c$. It follows from item (i) of Lemma \ref{free-perf} that either $(Z')^c\subseteq Z^c$, and hence $Z\wed Z'=Z$,  or $Z'\subseteq Z^c$, and we get a contradiction, because $C_i$ is contained in $\ol{Z'\setminus Z^c}$, and we are done. 
 Arguing in the same fashion and using that $C_1$ is contained in $\ol{(Z')^c\setminus Z}$, we are done if either 
 $\#(\term_{(Z')^c}\cap  Z^c)=0$, or  $\#(\term_{Z^c}\cap  (Z')^c)$ is in $\{0,3\}$.

Suppose either $\#(\term_{(Z')^c}\cap  Z^c)=2$, or $\#(\term_{Z^c}\cap  (Z')^c)=2$.  
Since $k_{Z^c\wedge (Z')^c}=k_{Z\cup Z'}\ge1$ and $k_{Z^c\cup (Z')^c}=k_{Z\wed Z'}\ge1$,  it follows from item (i) of Lemma \ref{3tail} that  $k_{Z\wed Z'}$ 
is in $\{2,3\}$. 
By contradiction, assume that $k_{Z\wed Z'}=2$.  Then item (i) of Lemma \ref{connect} implies that $Z\wed Z'$ is a 2-tail. 
Since $C_i\cup C_j$ and $C_1$ are contained respectively in $Z\wed Z'$ and $(Z\wedge Z')^c$, it follows from Corollary \ref{maxT2} that there is a $(Z\wed Z')$-terminal tail $W$ in $\T^2_{i,j}$. 
Therefore we get 
$$\emptyset\ne \term_W\cap\term_{Z\wed Z'}\subseteq \term_W\cap (\term_Z\cup\term_{Z'}),$$ 
 In this way, at least one between $Z$ and $Z'$ is 
not $\T^2_{i,j}$-free, yielding a contradiction.

Suppose $\#(\term_{(Z')^c}\cap  Z^c)=\#(\term_{Z^c}\cap  (Z')^c)=1$. It follows from item (ii) of Lemma \ref{3tail} that $k_{Z\cup Z'}=k_{Z^c\wed (Z')^c}=2$. 
Hence item (ii) of Lemma \ref{connect}  implies that 
$Z\cup Z'$ is a 2-tail. Since $C_i\cup C_j$ and $C_1$ are contained respectively in $Z\cup Z'$ and $(Z\cup Z')^c$, it follows from Corollary \ref{maxT2} that there is a $(Z\cup Z')$-terminal tail $W$ in $\T^2_{i,j}$. 
Arguing as in Case 2, at least one between $Z$ and $Z'$ is not $\T^2_{i,j}$-free, again a contradiction. The proof of the proposition is now complete. 
\end{proof}

Let $C$ be a nodal curve with irreducible components $C_1,\dots, C_p$ and let $(i,j)$ be in $\{1,\dots,p\}^2$. 
 If $\S^3_{i,j}$ is nonempty, set $W^3_0:=\wedge_{Z\in \S^3_{i,j}} Z$.  
It follows from Proposition \ref{S3} that $W^3_0$ is in $\S^3_{i,j}$. 
For every positive integer $m$, define inductively
$$\S^3_{i,j,m}:=\{Z : Z\in \S^3_{i,j} \text{ and } W^3_{m-1}\lhd Z\}.$$
and, if $\S^3_{i,j,m}$ is nonempty, we let
$$W^3_m:=\wed_{Z\in \S^3_{i,j,m}} Z.$$
Arguing as for the set of nested 2-tail of $C$, the tail $W^3_m$ is in $\S^3_{i,j,m}$, for every $m$ such that $\S^3_{i,j,m}$ is nonempty.
Let $N$ be the maximum positive integer such that $\S^3_{i,j,N}$ is non empty, and consider the set 
$$\T^3_{i,j}:=\{W^3_0,\dots W^3_{N}\}.$$
  We call $\T^3_{i,j}$ \emph{the set of nested 3-tail of $C$ with respect to $(i, j)$}.

For every $(i,j)$ in $\{1,\dots,p\}^2$,  we set 
\begin{equation}\label{nestik}
\T_{i,j}:=\T^1_{i,j}\sqcup \T^2_{i,j}\sqcup \T^3_{i,j}.
\end{equation}  
where $\T^1_{i,j}:=\T^1_i\sqcup\T^1_j$.
We call $\T_{i,j}$ \emph{the set of nested tails of $C$ with respect to $(i, j)$}.

\begin{Cor}\label{maxT3} 
Let  $C$ be a nodal curve with irreducible components $C_1,\dots,C_p$.  Let $\T^2_{i,j}$ and $\T^3_{i,j}$ be the  sets of nested 2-tails and 3-tails of $C$ with respect to $(i, j)$, where $(i, j)$ is in $\{1,\dots,p\}^2$.  If $Z$ is a 3-tail of $C$ such that $C_i\cup C_j$ and $C_1$ are contained respectively in $Z$ and $Z^c$, then there is a $Z$-terminal tail $W$ in $\T^2_{i,j}\cup \T^3_{i,j}$; if $k_W=3$, then $W$ is contained in $Z$.
\end{Cor}

\begin{proof}
If $Z$ is not $\T^2_{i,j}$-free, then we are done.
If $Z$ is $\T^2_{i,j}$-free, then, by the definition of $\T^3_{i,j}$, the tail $W^3_0$ is contained in $Z$. Arguing as in the proof of Corollary  \ref{maxT2}, the maximal tail of $\T^3_{i,j}$ contained in $Z$ is $Z$-terminal, and hence we are done.
\end{proof}

\section{Further results on tails of nodal curves}\label{sec5}

\noindent
Throughout this section, $C$ will be a nodal curve with irreducible components $C_1,\dots,C_p$, and 
$\T^2_{i,j},\T^3_{i,j}$ will be respectively the sets of nested 2-tails and 3-tails of $C$ with respect to $(i, j)$, for $(i, j)$ in $\{1,\dots,p\}^2$.

\begin{Lem}\label{res1} 
Let $Z$ be a tail of $C$ with $k_Z\ge 3$. Then there are no $Z$-terminal tails $W$ in $\T^2_{i,j}$ and $W'$ in $\T^3_{i,j}$ such that $W\cup W'\subseteq Z$. 
 \end{Lem}

\begin{proof}
Suppose by contradiction that there are  $Z$-terminal tails $W$ in $\T^2_{i,j}$ and $W'$ in $\T^3_{i,j}$ such that $W\cup W'\subseteq Z$.  
Set $X:=W\wed W'$.  It follows from item (iii) of Lemma \ref{connect} that $X$ is a tail with $k_X$ equal to $2$ or $3$. We distinguish two cases.

\smallskip

\emph{Case 1.} Assume that $X$ is a 2-tail. Since 
$C_i\cup C_j$ and $C_1$ are contained respectively in $X$ and $X^c$, it follows from Corollary \ref{maxT2} that there is a $X$-terminal tail $W''$ in $\T^2_{i,j}$ contained in $X$. 
 
We claim that $W''$ is different from $W$. Since $W'$ is contained in $Z$, it follows that $\term_{Z}\cap W'\subseteq \term_{W'}$, and hence $\term_Z\cap\term_{W}\cap W'\subseteq \term_W\cap\term_{W'}=\emptyset$,  where the last equality holds because $(W, W')$ is free. On the other hand, $W$ is $Z$-terminal, hence  
\begin{equation}\label{empt}
\term_Z\cap\term_{W}\subset W\setminus W'.
\end{equation}
The left hand side of~\eqref{empt} is nonempty, hence $W''\subseteq X\subsetneq W$, and the claim follows. 

Using (\ref{term}), we have
\begin{equation}\label{empt1}
 \term_{W''}\cap\term_X\subseteq \term_{W''}\cap (\term_W\cup\term_{W'}).
\end{equation}
Since the left hand side of~\eqref{empt1} is nonempty, we have that either $(W, W'')$, or $(W',W'')$ are terminal, which is a contradiction to the definition of $\T^2_{i,j}$ and $\T^3_{i,j}$.

\smallskip

\emph{Case 2.}  Assume that $X$ is a 3-tail.  It follows from item (iii) of Lemma \ref{connect} that $k_{W\cup W'}=2$. Furthermore, it follows from item (iv) of Lemma \ref{term-lem} that
\begin{equation}\label{eq1}
\term_Z\cap (\term_W\cup \term_{W'})\subseteq \term_{W\cup W'}.
\end{equation} 
The left hand side of (\ref{eq1}) has cardinality at least 2, because $W$ and $W'$ are $Z$-terminal and $(W, W')$ is free. Therefore, the equality holds in (\ref{eq1}), because $k_{W\cup W'}=2$. It follows that 
 $\term_{W\cup W'}$ is a desconnecting subset of $\term_Z$ of cardinality 2, which is  a contradiction because $Z$ is a tail such that $k_Z\ge 3$.
\end{proof}

\begin{Lem}\label{res2}
Let $Z$ be a tail of  $C$ with $k_Z\ge 4$. There are no $Z$-terminal tails $W$ in $\T^2_{i,j}$ and $W'$ in $\T^3_{i,j}$ such that $Z\subseteq W\wedge W'$ and $\#(\term_{W'}\cap \term_Z)=2$.
\end{Lem}

\begin{proof}
Suppose by contradiction that there are $Z$-terminal tails $W$ in $\T^2_{i,j}$ and $W'$ in $\T^3_{i,j}$ such that $Z\subseteq W\wedge W'$ and  $\#(\term_{W'}\cap \term_Z)=2$. Set $X:=W\wed W'$.  It follows from item (iii) of Lemma \ref{connect} that $X$ is a tail such that $k_X$ is $2$ or $3$, and from item (iii) of Lemma \ref{term-lem} that
\begin{equation}\label{eq2}
\term_Z\cap (\term_W\cup \term_{W'})\subseteq \term_X.
\end{equation} 
The left hand side of (\ref{eq2}) is a set of cardinality at least 3, because $(W, W')$ is free. Since $k_X\le 3$, the equality holds in (\ref{eq2}) and hence $X$ is a 3-tail.  It follows that $\term_X$ is a desconnecting subset of $\term_Z$ of cardinality 3,  which is a contradiction because $Z$ is a tail with $k_Z\ge 4$.
\end{proof}

\begin{Lem}\label{res3}
Let $Z$ be a tail of  $C$. There are no distinct  tails $W$ and $W'$ of $C$, with $k_W=2$ and $k_{W'}=3$,    
and satisfying the following conditions
\begin{itemize}
\item[(i)]
the pair $(W,W')$ is free;
\item[(ii)]
the tail $Z$ is contained in $W\wed W'$;
\item[(iii)]
the set of terminal points of $Z$ is contained in $\term_W\cup \term_{W'}$;
\item[(iv)]
the tail $Z$ is different from $W\wedge W'$ and $k_{W\cup W'}\ge1$.
\end{itemize}
\end{Lem}

\begin{proof}
Suppose that there are distinct tails $W$ and $W'$, with $k_W=2$ and $k_{W'}=3$, and  contradicting the statement of the lemma.
Notice that  $k_{W\wedge W'}\ge1$ and $k_{W\cup W'}\ge1$.  Set $X:=W\wed W'$.  Since $(W, W')$ is free, it follows from item (iii) of Lemma \ref{connect} that $X$ is a tail. On the other hand, item (iii) of Lemma \ref{term-lem} implies that the following inclusion holds
\begin{equation}\label{eq3}
\term_Z\cap (\term_W\cup\term_{W'})\subseteq \term_X.
\end{equation}
 Since the set of terminal points of $Z$ is contained in 
 $\term_W\cup \term_{W'}$,  the left hand side of (\ref{eq3}) is equal to $\term_Z$. It follows that the set of terminal points of $Z$ is a desconnecting subset of $\term_X$,  hence $\term_Z$ is equal to $\term_X$, because $X$ is a tail.  Since $Z\subseteq X$, we get that $Z$ is equal to $X$, which  is a contradiction.
\end{proof}

\begin{Lem}\label{res4}
Let $Z$ and $W$ be tails of $C$ such that $k_Z=4$ and  $W$ is in $\T^3_{i,j}$. Assume $W$  contained in $Z$ and 
$\#(\term_Z\cap \term_W)=2$. The following properties hold
 \begin{enumerate} [(i)]
\item  \label{res4(i)}
 there are no tails $W'$ in $\T^3_{i,j}$ and $W''$ in $\T^2_{i,j}$ such that 
 $Z\subseteq W'\wed W''$ and 
  $$\#(\term_Z\cap\term_{W'})=\#(\term_Z\cap \term_{W''})=1.$$ 
\item \label{res4(ii)}
 there are no tails $W'$ in $\T^3_{i,j}$ such that $Z\subseteq W'$ and 
 $\#(\term_{Z}\cap\term_{W'})=2.$
\end{enumerate}
 \end{Lem}

\begin{proof}
Write $\term_Z=\{R,S,T,U\}$ and $\term_W=\{R,S,V\}$, with $V$ different from $T$ and $U$.
Suppose there are $W'$ in $\T^3_{i,j}$ and $W''$ in $\T^2_{i,j}$ contradicting the statement of item (i). Notice that $W$ is different from $W'$ and $W''$.  Write
 $\term_{W'}=\{T,F,G\}$ and $\term_{W''}=\{U,H\}$, where $\{F,G,H\}\cap \term_Z$ is empty.
 Set $X_1:=W'\wed W''$, $X_2:=X_1\wed Z^c$ and $X_3:=X_2\cup W$. 
 
 We claim that $X_3$ is a 2-tail. It follows from item (iii) of Lemma \ref{term-lem}
  that $T$ and $U$ are terminal points of $X_1$, and from item (iii) of Lemma \ref{connect} that $X_1$ is a tail such that $k_{X_1}$ is $2$ or $3$.  Notice that the terminal points of $X_1$ are not $T$ and $U$, otherwise $\term_{X_1}$ would be a desconnecting subset of $\term_Z$ of cardinality 2,  which is not possible because $Z$ is a 4-tail. Therefore, $X_1$ is a 3-tail with $T$, $U$ and $K$ as terminal points, where $K$ is in $\{F,G,H\}$.  
 It follows from item (i) of Lemma \ref{term-lem} that the terminal points of $X_2$ are $R$, $S$ and $K$,  hence item (i) of Lemma \ref{connect} implies that $X_2$ is a 3-tail. Moreover, it follows from item (ii) of Lemma  \ref{term-lem} that the terminal points of $X_3$ are $V$ and $K$, and  
 item (ii) of Lemma \ref{connect} implies that $X_3$ is a 2-tail, concluding the proof of the claim. 
 
Notice that $C_i\cup C_j$ and $C_1$ are contained respectively in  $X_3$ and $X_3^c$. Since $X_3$ is a 2-tail, it follows from  Corollary \ref{maxT2} that there is a tail $\wh W$ in $\T^2_{i,j}$ contained in $X_3$ and such that
\begin{equation}\label{VK}
\term_{\widehat W}\cap \{V, K\}\ne\emptyset
\end{equation} Notice that $T\in \term_{X_1}\setminus\term_{X_2}$. In particular, $T$ is not in $X_2$, because $X_2$ is contained in $X_1$. Similarly, we have that $T$ is not in $W$. In this way,  we obtain 
 $T\in W''\setminus X_3\subseteq W''\setminus \widehat W$, hence $W''$ is different from $\wh W$. Thus,
  $(W'', \widehat W)$ is free, and hence it follows from (\ref{VK}) that either $(\widehat W, W)$ or $(\widehat W, W')$ is terminal, which is a contradiction. This completes the proof of the item (i).

Suppose now there is $W'$ in $\T^3_{i,j}$ contradicting the statement of item (ii).  Notice that $W$ is different form $W'$. Write $\term_{W'}=\{T,U,F\}$, with $F$ different from $R$ and $S$.  Define  $X_1:=W'\wedge Z^c$ and $X_2:=X_1\cup W$. 
It follows from item (i) of Lemma  \ref{term-lem} that $R$, $S$ and $F$ are the terminal points of $X_1$ and from item (i) of Lemma \ref{connect} that $X_1$ is a 3-tail. Using item (ii) of Lemma   \ref{term-lem}, we have that $V$ and $F$ are the terminal points of $X_2$, hence item (ii) of Lemma \ref{connect} implies that 
$X_2$ is a 2-tail.  Notice that $C_i\cup C_j$ and $C_1$ are contained respetively in $X_2$ and $X_2^c$, hence it follows from Corollary \ref{maxT2} that there is a $X_2$-terminal tail $\widehat W$ in $\T^2_{i,j}$.
We conclude that at least a pair between $(W, \widehat W)$ and $(W', \widehat W)$ is terminal, which is a contradiction.
\end{proof}

\begin{Lem}\label{res5} 
Let $Z$ be a $\T^2_{i,j}$-normalized tail of $C$ such that $k_Z$ is in $\{2,3,4\}$ and $C_i\cup C_j\subseteq Z^c$. Let $W$ be a tail in $\T^2_{i,j}\cup \T^3_{i,j}$ such that $Z\subseteq W$ and such that
$$
\#(\term_Z\cap\term_W)=
\begin{array}{ll}
\begin{cases}
1 & \text{ if } k_Z=2;\\
2 & \text{ if } k_Z \text{ is } 3 \text{ or }4.
\end{cases}
\end{array}
$$  
Then there is a $Z$-terminal tail $W'$ in $\T^2_{i,j}\cup\T^3_{i,j}$ contained in $Z^c$.
 \end{Lem}

\begin{proof}
  Notice that if $k_W=2$, then $k_Z=2$. In fact, if $k_W=2$ and $k_Z$ is $3$ or $4$, then 
   $\#(\term_Z\cap\term_W)=2$, and hence $\term_W$ is a desconnecting subset of cardinality 2 of $\term_Z$, which is a contradiction, because $Z$ is a tail with $k_Z\ge3$.
  
    Set $\A:=\term_Z$, $\B:=\term_W$ and $X:=Z^c\wed W$. 
   It follows from item (i) of Lemma  \ref{term-lem} that $\term_X=\Diff(\A,\B)$, and hence $k_X$ is $2$ or $3$. Therefore,  item (i) of
   Lemma \ref{connect} implies that $X$ is a tail. Since $C_i\cup C_j$ and $C_1$ are contained respectively in $X$ and $X^c$,  it follows from  Corollaries \ref{maxT2} and \ref{maxT3} that there is a $X$-terminal tail  $W'$ in $\T^2_{i,j}\cup \T^3_{i,j}$. 
   Moreover, if $k_X=2$, we can choose $W'$ such that $k_{W'}=2$ and $W'\subseteq X$, while if $k_{W'}=3$, we can choose $W'$ such that $W'\subseteq X$.  
  
  We distinguish three cases. 
  In the first case, we have $k_W=2$. It follows that  $k_Z=2$, then $k_X=\#\Diff(\A,\B)=2$, and hence  $k_{W'}=2$ with $W'\subseteq X\subseteq Z^c$. In particular, $W$ is different from $W'$, and hence
  $(W,W')$ is free. Since $W'$ is $X$-terminal, using (\ref{term}) we see that $W'$ is $Z$-terminal, and we are done.  
  In the second case, we have $k_{W'}=3$. It follows that $W'\subseteq X\subseteq Z^c$, then $W$ is different form $W'$.  As in the first case, we see that $W'$ is $Z$-terminal, and we are done.
In the third case, we have $k_W=3$ and $k_{W'}=2$. Again, $W$ is different from $W'$, hence   as in the first case, we see that $W'$ is $Z$-terminal.
   Since $Z$ is $\T^2_{i,j}$-normalized, one of the following conditions holds 
   $$W'\subseteq Z, \,\,\, Z\subseteq W', \,\,\, Z^c\subseteq W', \,\,\, W'\subseteq Z^c.$$  
   The first possibility does not hold, because $C_i$ is contained in $\ol{W'\setminus Z}$. Assume that $Z\subseteq W'$: If $k_Z=4$, then we get a contradiction to Lemma \ref{res2}, while if $k_Z$ is $2$ or $3$, we get a contradiction to  Lemma \ref{res3}. The third possibility does not hold, because $C_1\subseteq W^c\subseteq Z^c$, while $C_1\subseteq (W')^c$. It follows that $W'\subseteq Z^c$, and we are done.
\end{proof}

\begin{Lem}\label{res6} 
Let $Z$ be a $\T^2_{i,j}$-normalized 3-tail of $C$ such that $C_i\cup C_j\cup C_1\subseteq Z^c$. Let $W$ be a $Z$-terminal tail in $\T^3_{i,j}$ contained in $Z^c$, with $\#(\term_Z\cap \term_W)=2$. Then there is a $Z$-terminal tail $W'$ in $\T^2_{i,j}$ such that $Z\subseteq W'$. 
 \end{Lem}

\begin{proof}
Set $\term_Z=\{R,S,T\}$ and $\term_W=\{R,S,U\}$, with $U$ different form $T$.  
It follows from item (ii) of Lemma \ref{term-lem} that $T$ and $U$ are the terminal points of  $Z\cup W$, hence item (ii) of Lemma \ref{connect} implies that $Z\cup W$ is a 2-tail.  Since $C_i\cup C_j$ and $C_1$ are contained respectively in $Z\cup W$ 
and $(Z\cup W)^c$, it follows from 
Corollary \ref{maxT2} that there is a $(Z\cup W)$-terminal tail $W'$ in $\T^2_{i,j}$ such that $W'\subseteq Z\cup W$. Since $(W, W')$ is free, it follows from (\ref{term}) that $W'$ is $Z$-terminal.  By the hypothesis, $Z$ is $\T^2_{i,j}$-normalized, hence one of the following conditions holds $$W'\subseteq Z, \,\,\, Z^c\subseteq W', \,\,\, W'\subseteq Z^c, \,\,\, Z\subseteq W'.$$ The first and the second possibility do not hold, because  $C_i\cup C_j$ and $C_1$ are contained respectively in $\ol{W'\setminus Z}$ and $\ol{Z^c\setminus W'}$. The third one does not hold as well, as we can see by applying Lemma \ref{res1} to $Z^c$. Thus, we have $Z\subseteq W'$, and we are done.
\end{proof}

\begin{Lem}\label{res7} 
Let $Z$ be a $\T^2_{i,j}$-normalized tail of $C$ such that $k_Z$ is in $\{3,4\}$ and 
$C_i\cup C_j\cup C_1\subseteq Z^c$. Let $W$ in $\T^2_{i,j}$ and $W'$ in $\T^3_{i,j}$ be  such that $Z\subseteq W\wed W'$ and 
$$\#(\term_Z\cap \term_W)=\#(\term_Z\cap \term_{W'})=1.$$ Then there is a $Z$-terminal tail  $W''$ in $\T^2_{i,j}\cup\T^3_{i,j}$  contained in $Z^c$.
 \end{Lem}

\begin{proof}
Set $X:=W\wed W'$ and $X':=X\wed Z^c$. Write
$$
\term_Z=\{R, S, T\} \text{ if } k_Z=3 \quad \text{ and \;\;} \term_Z=\{R, S, T, U\} \text{ if } k_Z=4 
$$
where $\term_W=\{R,V\}$ and $V$ is not a terminal point of $Z$, and where the intersection of 
$\term_Z$ and $\term_{W'}$ consists of  $S$.

 It follows from items (iii) of Lemmas \ref{term-lem} and  \ref{connect} that $X$ is a tail with $R$ and $S$ as terminal points and such that $k_X$ is $2$ or $3$. In particular, $k_X=3$, otherwise $\{R,S\}$ would be a proper desconnecting subset of $\term_Z$, which is a contradiction. Thus, we can write $\term_X=\{R, S, K\}$,  where $K$ is in $\term_W\cup\term_{W'}$ and $K$ is different from $R$ and $S$.

 Using again item (i) of Lemma \ref{term-lem} we have that $k_{X'}$ is $2$ or $3$.
 Therefore, it follows from item (i) of Lemma \ref{connect} that $X'$ is a tail.  Notice that 
 $C_i\cup C_j$ and $C_1$ are contained respectively in $X'$ and $(X')^c$, hence Corollaries \ref{maxT2} and \ref{maxT3} imply the existence of a $X'$-terminal tail $W''$ in $\T^2_{i,j}\cup \T^3_{i,j}$. Moreover, if $k_{X'}=2$ or $k_{W''}=3$, we can choose $W''$ such that $W''\subseteq X'$, and hence  we are done in this case. Thus, we can assume $k_{X'}=3$ and $k_{W''}=2$. In particular, we have $k_Z=4$ and that $T$, $U$, $K$ are the terminal points of $X'$.

Notice that $W$ is different from $W''$, otherwise $W$ should be $X'$-terminal, and hence $V=K$, which implies that $\term_W\subsetneq\term_X$: This is a contradiction because $X$ is a tail. In particular, either  
 $W\lhd W''$, or $W''\lhd W$, because $W$ and $W''$ are tails in $\T^2_{i,j}$. 
Since $W''$ is $X'$-terminal and $(W,W'')$ and $(W',W'')$ are free, it follows that $W''$ is $Z$-terminal. The tail $Z$ is $\T^2_{i,j}$-normalized, then one of the following conditions holds 
$$Z\subseteq W'', \,\,\, W''\subseteq Z, \,\,\, Z^c\subseteq W'', \,\,\, W''\subseteq Z^c.$$ 
If $Z\subseteq W''$, then $Z\subseteq W\wed W''$, and hence either $Z\subseteq W\lhd W''$, or $Z\subseteq W''\lhd W$, which is a contradiction because $W$ and $W''$ are $Z$-terminal.
The second and the third case do not hold,  because $C_i\cup C_j$ and $C_1$ are contained respectively in $\ol{W''\setminus Z}$ and $\ol{Z^c\setminus W''}$. It follows that $W''\subseteq Z^c$, and we are done.
\end{proof}

\section{Extending the degree 2 Abel--Jacobi map}\label{ext-sec}

\noindent
Throughout this section, we fix a regular local smoothing $f\: \C\ra B$ of a nodal curve $C$ with irreducible components $C_1,\dots, C_p$, where $B$ is the spectrum of a Henselian DVR with residue field $k$ and quotient field $K$. We let $\mathcal E$ be the canonical polarization on $\C/B$ and $\sigma\:  B\ra \C$ be a section of $f$ through the $B$-smooth locus of $\C$ such that $\sigma(Spec(k))$ is contained in $C_1$.  Moreover, 
 for $(i,j)$ in $\{1,\dots,p\}^2$, let $\T_{i,j}$ be the set of nested tails of $C$ with respect to $(i,j)$  defined in (\ref{nestik}); we set  $\Sigma:=\sigma(B)$ and
\begin{equation}\label{LDEF}
\O_{\T_{i,j}}:=\O_{\C}\left(-\sum_{Z\in \T_{i,j}}Z\right)\otimes\O_C
\end{equation}

 Let $J_{\C_K}$ be the Jacobian of the generic fiber $\C_K$ of $f$. 
Recall that the \emph{N\'eron model of $J_{\C_K}$} is a $B$-scheme $N(J_{\C_K})$, smooth and separated over $B$, whose generic fiber is isomorphic to  $J_{\C_K}$ and uniquely determined by the following universal property (the \emph{N\'eron mapping property}): for every $B$-smooth scheme $Z$ with generic fiber $Z_K$ and for every $K$-morphism $u_K\:  Z_K\ra J_{\C_K}$, there is a unique extension of $u_K$ to a morphism $u\:  Z\ra N(J_{\C_K})$ (for more details on N\'eron models, see \cite{BLR}). 

The Jacobian $J_{\C_K}$ is an open subset of Esteves's compactified Jacobian $J^\sigma_{\mathcal E}$ introduced in Section \ref{intro-sec}. The following result, due to  Busonero, Kass and Melo--Viviani, states a relationship between $J^\sigma_{\mathcal E}$ and $N(J_{\C_K})$.

\begin{Thm}\label{Ner}
The $B$-smooth locus of $J^\sigma_{\mathcal E}$ is isomorphic to the N\'eron model of $J_{\C_K}$.
\end{Thm}

\begin{proof}
See \cite{B}, \cite[Theorem A]{K} and  \cite[Theorem 3.1]{MeVi}.
\end{proof}

Let $\dot{\C}^2$ be the $B$-smooth locus of $\C\times_B\C$.
Since $\dot\C^2$ is $B$-smooth, combining the N\'eron mapping property and Theorem \ref{Ner}, we obtain a natural extension
\begin{equation}\label{A-N}
\alpha^2_{\I_{\Sigma|\C}^{-2}}\:  \dot{\C}^2\lra J^\sigma_{\mathcal E}
\end{equation}
 of the degree 2 Abel--Jacobi map $\alpha^2_{\I_{\Sigma|\C}^{-2},K}$ defined in (\ref{AJ-generic}).  As in \cite{CE}, we call $\alpha^2_{\I_{\Sigma|\C}^{-2}}$ the \emph{Abel--N\'eron} map. Although the definition of the Abel--N\'eron map is natural, it is not explicit. To get a modular description of $\alpha^2_{\I_{\Sigma|\C}^{-2}}$, we need the following Lemma.

\begin{Prop}\label{quasistab}
Let $C$ be a nodal curve with irreducible components $C_1,\dots,C_p$. Fix smooth points $P$, $Q$ and $Q'$ of $C$ contained respectively in $C_1$, $C_i$ and $C_j$, where $(i,j)$ is in $\{1,\dots,p\}^2$. Then $\O_C(2P-Q-Q')\otimes \O_{\T_{i,j}}$  
is $C_1$-quasistable if and only if  $\O_C(2P-Q-Q')\otimes \O_{\T_{i,j}}$ is $C_1$-quasistable at every $\T_{i,j}$-normalized tail of $C$.
\end{Prop}

\begin{proof}
The ``only if'' part of the statement is trivial. Let us prove the ``if'' part of the statement.  
Set $L:=\O_C(2P-Q-Q')\otimes \O_{\T_{i,j}}$. 

\smallskip

\emph{First Step}.
Suppose that $Z$ is a  $\T^3_{i,j}$-normalized subcurve of $C$ with connected components $Z_1,\dots,Z_c$, where $c\ge 1$. 

For every $s$ in $\{1,\dots,c\}$, we claim that $Z_s$ is $\T^3_{i,j}$-normalized. 
Indeed, assume that $Z_s$ is $Z'$-terminal, where $Z'$ is in $\T^3_{i,j}$.   Then $Z$ is  $Z'$-terminal, and hence one of the following conditions holds
$$Z'\subseteq Z,\,\,\,   Z^c\subseteq Z',  \,\,\, Z\subseteq Z', \,\,\, Z'\subseteq Z^c.$$
If $Z'\subseteq Z$ (respectively $Z^c\subseteq Z'$), then either $Z'\subseteq Z_s$ or $Z'\subseteq Z_s^c$ (respectively either $(Z')^c\subseteq Z_s$ or $(Z')^c\subseteq Z_s^c$), because $Z'$ is a tail. If $Z\subseteq Z'$ (respectively $Z'\subseteq Z^c$), then $Z_s\subseteq Z'$ (respectively $Z'\subseteq Z_s^c$).  In any case, $(Z_s, Z')$ is perfect.

\smallskip
\emph{Second Step.} Suppose that $L$ is $C_1$-quasistable at every $\T^3_{i,j}$-normalized tail of $C$. We claim that $L$ is $C_1$-quasistable at every $\T^3_{i,j}$-normalized subcurve of $C$.

By contradiction, assume that $L$ is not $C_1$-quasistable at a $\T^3_{i,j}$-normalized subcurve $Z$. In particular, $Z$ is not a tail and $L$ is not $C_1$-quasistable at least at one connected component of $Z$. If $Z_1,\dots, Z_c$ are the connected components of $Z$,  where $c\ge 1$, then we can assume that $L$ is not $C_1$-quasistable at $Z_1$. It follows from the first step that $Z_1$ is $\T^3_{i,j}$-normalized.
Let $Y_1,\dots,Y_d$ be the connected components of $Z_1^c$, where $d\ge 1$. Notice that $Y_1,\dots,Y_d$ are tails of $C$.  Since $Z_1$ is $\T^3_{i,j}$-normalized and connected and since $L$ is not $C_1$-quasistable at $Z_1$, we have that $Z_1^c$ is $\T^3_{i,j}$-normalized and 
$d\ge 2$. In particular,  it follows again from the first step that $Y_t$ is $\T^3_{i,j}$-normalized,
 for every $t$ in $\{1,\dots, d\}$.  Since $Y_t$ is a tail of $C$,  we get that $L$ is $C_1$-quasistable at $Y_t$, for every $t$ in $\{1,\dots, d\}$, which implies that $L$ is $C_1$-quasistable at $Z_1^c$. It follows that $L$ is semistable and not $C_1$-quasistable at $Z_1$, and hence $\beta_L(Z_1)=0$ and $C_1\subseteq Z_1$. 
 Recall that $Y_1^c$ is a tail and $Y_1^c=Z_1\cup Y_2\dots\cup Y_d$. Since  
 $Y_1$ is $\T^3_{i,j}$-normalized, we get that $Y_1^c$ is $\T^3_{i,j}$-normalized. In particular, $L$ is $C_1$-quasistable at $Y_1^c$, and  hence $\beta_{L}(Y^c_1)>0$, because $C_1\subseteq Y^c_1$.  
 
 On the other hand, the condition $\beta_L(Z_1)=0$ implies    
$\deg_L(Z_1^c)=k_{Z_1}/2$. Since $Y_1,\dots,Y_d$ are the connected components of $Z_1^c$, it follows that $\sum_{t=1}^d \deg_LY_t=\sum_{t=1}^d k_{Y_t}/2$. Since $L$ is $C_1$-quasistable at every $Y_t$, we have $\deg_L(Y_t)\le k_{Y_t}/2$, hence necessarily  
$\deg_L(Y_t)=k_{Y_t}/2$, for every $t$ in $\{1,\dots, d\}$.  In this way we obtain
$$\beta_L(Y^c_1)=\deg_L(Z_1)+\sum_{t=2}^d \deg_L(Y_t)+k_{Z_1}/2-\sum _{t=2}^d k_{Y_t}/2$$$$=\beta_L(Z_1)+\sum_{t=2}^d (\deg_L(Y_t)-k_{Y_t}/2)=0$$
which is a contradiction.

\smallskip

\emph{Third Step}. Suppose that $L$ is $C_1$-quasistable at every $\T^3_{i,j}$-normalized tail of $C$. 
We want to show that $L$ is $C_1$-quasistable at every subcurve of $C$ containing $C_1$.

Assume by contradiction that there are subcurves of $C$ containing $C_1$ at which $L$ is not $C_1$-quasistable. Let $Z_0\subsetneq C$ be a maximal proper subcurve containing $C_1$ and such that $\beta_L(Z_0)\le 0$. 
 The second step implies that $Z_0$ is not $\T^3_{i,j}$-normalized.  Let $\widehat{W}$ be the maximal tail of $\T^3_{i,j}$ such that $(Z_0, \widehat{W})$ is terminal and not perfect. In particular, $Z_0$ does not contain $\widehat{W}^c$, and $\widehat{W}^c$ does not contain $Z^c_0$. Therefore, we have $Z_0\subsetneq \widehat{W}^c\cup Z_0\ne C$, and hence $\beta_L(\widehat{W}^c\cup Z_0)>0$, by the maximal property of $Z_0$. 
Notice that $C_1$ is contained in $\widehat{W}^c\wed Z_0$. 

 We claim that $\widehat{W}^c\wed Z_0$ is $\T^3_{i,j}$-normalized. Assume by contradiction that there is a tail $W'$ in $\T^3_{i,j}$ such that $(\widehat{W}^c\wed Z_0, W')$ is terminal and not perfect. Notice that
  $\widehat{W}$ is different form $W'$, because $(\widehat{W}^c\wed Z_0, \widehat{W})$ is perfect, and hence $(\widehat W,W')$ is not terminal. It follows that $(Z_0, W')$ is terminal. Moreover, by the definition of $\T^3_{i,j}$,  either $W'\subsetneq \widehat{W}$, or $\widehat{W}\subsetneq W'$. In the first case, we have $\widehat{W}^c\wed Z_0\subseteq \widehat{W}^c\subseteq (W')^c$, hence $(\widehat{W}^c\wed Z_0, W')$ is perfect, which is a contradicion. In the second case,  
 $(Z_0, W')$ is perfect, by the maximal property of $\widehat{W}$, then one of the following conditions holds
$$
W'\subseteq Z_0, \,\,\, W'\subseteq Z^c_0, \,\,\, Z_0\subseteq W', \,\,\, Z_0^c\subseteq W'.
$$
If $W'\subseteq Z_0$ (respectively $W'\subseteq Z^c_0$), then $\widehat{W}\subseteq Z_0$ (respectively $\widehat{W}\subseteq Z_0^c$), implying that $(Z_0,\widehat{W})$ is perfect, a contradiction. 
If $Z_0\subseteq W'$ (respectively $Z_0^c\subseteq W'$), then $\widehat{W}^c\wed Z_0\subseteq Z_0\subseteq W'$ (respectively $(W')^c\subseteq \widehat{W}^c\wed Z_0$), and hence $(\widehat{W}^c\wed Z_0, W')$ is perfect, again a contradiction. The proof of the claim is complete.

 Since $C_1$ is contained in $\widehat{W}^c\wed Z_0$ and $\widehat{W}^c\wed Z_0$ is $\T^3_{i,j}$-normalized, it follows from the second step that $\beta_L(\widehat{W}^c\wed Z_0)>0$. By \cite[Lemma 3]{E01} and recalling that $\beta_L(W)\in \ze+1/2\cdot\ze$, for every subcurve $W$ of $C$,  we have 
 $$\beta_L(Z_0)\ge \beta_L(\widehat{W}^c\cup Z_0)+\beta_L(\widehat{W}^c\wed Z_0)-\beta_L(\widehat{W}^c)\ge 1-\beta_L(\widehat{W}^c).$$ 
 By the definion of $L$ and since $k_{\widehat{W}^c}=3$, we have 
 \begin{equation}\label{betaWc}
 \beta_L(\widehat{W}^c)=\deg_L(\widehat{W}^c)+k_{\widehat{W}^c}/2=\deg_{\O_C(2P-Q-Q')}(\widehat{W}^c)-k_{\widehat{W}^c}/2=1/2,
 \end{equation}
and hence $\beta_L(Z_0)>0$, which is a contradiction.

 \smallskip 
 \emph{Fourth Step}.  Suppose that $L$ is $C_1$-quasistable at every $\T^3_{i,j}$-normalized tail of $C$. We want to show that $L$ is $C_1$-quasistable.
 
 Assume by contradiction that $L$ is not $C_1$-quasistable and let $Z_0\subsetneq C$ be a maximal proper subcurve among the subcurves of $C$ at which $L$ is not $C_1$-quasistable.  It follows from the third step that 
 $C_1$ is not contained in $Z_0$ and hence $\beta_L(Z_0)<0$. 
 Moreover, the second step implies that $Z_0$ is not  $\T^3_{i,j}$-normalized. Let $\widehat W$ be the maximal tail of $\T^3_{i,j}$ such that 
 $(Z_0, \widehat W)$ is terminal and not perfect. In particular, $\widehat{W}^c$ does not contain $Z_0^c$, and hence $\widehat{W}^c\cup Z_0$ is strictly contained in $C$. Since 
 $C_1$ is contained in $\widehat{W}^c\cup Z_0$, again by the third step we have $\beta_L(\widehat{W}^c\cup Z_0)\ge 1/2$. 
 It follows from (\ref{betaWc}) that $\beta_L(\widehat{W}^c)=1/2$. Since $k_{W_1\cup W_2}\le k_{W_1}+k_{W_2}$, for any subcurves $W_1$ and $W_2$ of $C$, we get 
 $$\beta_L(Z_0)\ge \beta_L(\widehat{W}^c\cup Z_0)-\beta_L(\widehat{W}^c)\ge0,$$ which is a contradiction.

\smallskip

\emph{Fifth Step}. Suppose that $L$ is $C_1$-quasistable at every $\T_{i,j}$-normalized tail of $C$. 
We want to show that $L$ is $C_1$-quasistable.  

Assume by contradiction that $L$ is not $C_1$-quasistable. It follows from the fourth step that there is a tail $Z$  of $C$  and  a tail $W$ in $\T^1_i\cup \T^1_j\cup \T^2_{i,j}$ such that 
$(W, Z)$ is terminal and not  perfect. If $(k_Z, k_W)=(1, 1)$, then $Z\in \{W, W^c\}$, hence $(W,Z)$ is perfect, a contradiction. If $(k_Z, k_W)\ne (1, 1)$, then using Lemma \ref{free-perf} we see that $k_Z\ge 2$ and $k_W=2$, and either $\term_W$ is contained in $\term_Z$, 
and hence $Z$ is equal to one of the two tails $W$ and $W^c$, or 
$\#(\term_Z\cap \term_W)=k_W-1$. In both cases $(W, Z)$ is perfect, which is a contradiction.
\end{proof}

We are ready to state and prove the main Theorem of the paper, containing a modular interpretation of the Abel--N\'eron map (\ref{A-N}).  
Recall that $f\: \C\ra B$ is a regular local smoothing of a nodal curve $C$ with irreducible components $C_1,\dots, C_p$, where $B$ is the spectrum of a Henselian DVR with residue field $k$. 
Keep the notation of Section \ref{intro-sec}. Let $\dot\C$ be the $B$-smooth locus of $\C$. The $B$-smooth locus of $\C^2=\C\times_B\C$ is  
 $\dot\C^2=\dot\C\times_B \dot\C$. 
 For every $\{i,j\}$ in $\{1,\dots,p\}^2$, set
\begin{equation}\label{Zij}
Z_{i,j}:=\sum_{W\in\T_{i,j}} W,
\end{equation}
where $\T_{i,j}$ is the set of nested tails of $C$ with respect to $(i,j)$, and 
$$\dot \C_{i,j}:=\dot\C^2\cap(C_i\times C_j).$$ 
An easy computation shows that $\dot\C^2\times_B\C$ is smooth. Since $\dot\C^2\times_B\C$ is smooth, it follows that $\dot\C_{i,j}\times Z_{i,j}$ is a Cartier divisor of $\dot\C^2\times_B\C$, for every $(i,j)$ in $\{1,\dots,p\}^2$. 
Consider the family of curves $\rho\: \dot{\C}^2\times_B\C\ra\dot\C^2$, where $\rho$ is the projection onto the first factor, and the invertible sheaf $\dot\L$ on $\dot{\C}^2\times_B\C/\dot\C^2$ defined as
\begin{equation}\label{Ldot}
\dot{\L}:=(\xi^*\I_{\Sigma|\C}^{-2}\otimes \I_{\Delta_1|\C^3}\otimes\I_{\Delta_2|\C^3})|_{\dot\C^2\times_B\C}\otimes\O_{\dot{\C}^2\times_B\C}\left(-\sum_{1\le i,j\le p}\dot\C^2_{i,j}\times Z_{i,j}\right).
\end{equation}

For every $(i,j)$ in $\{1,\dots,p\}^2$, recall the definition of $\O_{\T_{i,j}}$ introduced in (\ref{LDEF}).

\begin{Thm}\label{main1}
Let $C$ be a nodal curve $C$ defined over an algebraically closed field $k$, with irreducible components $C_1,\dots, C_p$. 
Let $f\: \C\ra B$ be a regular local smoothing of $C$, where $B$ is the spectrum of a Henselian DVR with residue field $k$.  Fix smooth points $P$, $Q$ and $Q'$ of $C$ contained respectively in $C_1$, $C_i$ and $C_j$, where $(i,j)$ is in $\{1,\dots,p\}^2$. If $\T_{i,j}$ is the set of nested tails of $C$ with respect to $(i,j)$, then the  invertible sheaf 
$$\O_C(2P-Q-Q')\otimes \O_{\T_{i,j}}$$
on $C$ is $C_1$-quasistable. In particular, let 
 $\sigma\:  B\ra \C$ be a section of $f$ through the $B$-smooth locus of $\C$ such that $\sigma(Spec(k))$ is contained in $C_1$ and set $\Sigma:=\sigma(B)$. If $\mathcal E$ is the canonical polarization on $\C/B$, then the Abel--N\'eron map $\alpha^2_{\I_{\Sigma|\C}^{-2}}\:  \dot{\C}^2\ra J^\sigma_{\mathcal E}$ is induced by the invertible sheaf $\dot\L$ on $\dot\C^2\times_B\C/\dot\C^2$ defined in (\ref{Ldot}).
\end{Thm}

\begin{proof}
Let $\C_K$ be the generic fiber of $f$.  We claim that the second statement of the theorem follows from the first one. Indeed, notice that   
$$\dot\L|_{\C_K\times_B \C_K\times_B\C}\simeq (\xi^*\I_{\Sigma|\C}^{-2}\otimes\I_{\Delta_1|\C^3}\otimes\I_{\Delta_2|\C^3})|_{\C_K\times_B \C_K\times_B\C}$$
and, for every  $(Q,Q')$ in $\dot\C_{i,j}$, we have 
$$
\dot\L|_{\rho^{-1}(Q,Q')}\simeq 
\O_C(2P-Q-Q')\otimes\O_{\T_{i,j}}
$$
where $\rho\: \dot{\C}^2\times_B\C\ra\dot\C^2$ is the projection onto the first factor and
 $P=\sigma(Spec(k))$. 
Therefore, if we prove the first statement of the theorem, then $\dot\L$ induces a morphism from $\dot\C^2$ to $J^\sigma_{\mathcal E}$ which is equal to $\alpha^2_{\I_{\Sigma|\C}^{-2}}$, because it coincides 
with $\alpha^2_{\I_{\Sigma|\C}^{-2}}$ over the open subset $\C_K\times_B \C_K$ of $\dot\C^2$ and  
$J^\sigma_{\mathcal E}$ is a separate scheme.

We prove now the first statement.
Let $(i,j)$ be in $\{1,\dots,p\}^2$ and let $P$, $Q$ and $Q'$ be smooth points of $C$ contained respectively in $C_1$, $C_i$ and $C_j$. We set $L:=\O_C(2P-Q-Q')\otimes\O_{\T_{i,j}}$. To show that $L$ is $C_1$-quasistable, 
we will use Lemma \ref{quasistab}. 
Let $Z$ be a  $\T_{i,j}$-normalized tail of $C$.  If $k_Z=1$ and $C_1\subseteq Z^c$, then $Z$ is in  $\T^1_i$ (respectively in $\T^1_j$) if and only if $C_i\subseteq Z$ (respectively $C_j\subseteq Z$), hence 
$\deg_Z L=0$ and $L$ is $C_1$-quasistable at $Z$. If $k_Z=1$ and $C_1\subseteq Z$, then $\deg_Z L=-\deg_{Z^c}L=0$,  and $L$ is $C_1$-quasistable at $Z$. 
 Therefore, we can assume $k_Z\ge 2$. Set 
$$t^+_Z:=\underset{W\subseteq Z  \text{ or } Z\subseteq W}{\sum_{W\in \T^2_{i,j}\cup\T^3_{i,j}}}\#(\term_Z\cap \term_W), \,\,\,\, t^-_Z:=\underset{W\subseteq Z^c  \text{ or } Z^c\subseteq W}{\sum_{W\in \T^2_{i,j}\cup\T^3_{i,j}}}\#(\term_Z\cap \term_W).$$
Since item (iii) of Lemma \ref{free-perf} implies that  $(Z, Z')$ is free, for every 1-tail $Z'$ of $C$, we have 
 $$\deg_Z L=\deg_Z \O_C(2P-Q-Q')+t^+_Z-t^-_Z.$$  

In the sequel, we will use several times the following observation following from the construction of the set $\T_{i,j}$: for each $s$ in $\{2,3\}$, there is at most one $Z$-terminal tail $W_1$ in $\T^s_{i,j}$ such that $W_1\subseteq Z$ and at most one $Z$-terminal tail $W_2$ in $\T^s_{i,j}$ such that $Z\subseteq W_2$. Similarly, there is at most one $Z$-terminal tail $W_3$ in $\T^s_{i,j}$ such that $W_3\subseteq Z^c$ and at most one $Z$-terminal tail $W_4$ in $\T^s_{i,j}$ such that $Z^c\subseteq W_4$.

 \smallskip
 
 \emph{Case 1}. Suppose that $C_i\cup C_j\subseteq Z$ and $C_1\subseteq Z^c$. 
 If $W$ is a $Z$-terminal tail in $\T^2_{i,j}\cup\T^3_{i,j}$, then either $W\subseteq Z$ or $Z\subseteq W$, because $Z$ is $\T_{i,j}$-normalized. In particular, we have $t^-_Z=0$, and $\deg_Z L=-2+t^+_Z$.
 
 If $k_Z$ is $2$ or $3$, then it follows from Corollary \ref{maxT2} and Corollary \ref{maxT3} that there is at least a $Z$-terminal tail contained in $\T^2_{i,j}\cup\T^3_{i,j}$. In particular, we have $1\le t^+_Z\le k_Z$, hence $-1\le \deg_Z L\le  -2+k_Z$, and in this way $-k_Z/2\le \deg_Z L<k_Z/2$.
 
 If $k_Z\ge 4$, then it follows from Lemma \ref{res1} and Lemma \ref{res2} that $t^+_Z\le 4$, where the equality holds if and only if one of the following conditions holds
 
 \begin{itemize}
\item[(i)]
there are $W,W'$ in $\T^3_{i,j}$ such that $W\subseteq Z\subseteq W'$ and 
$\#(\term_Z\cap \term_W)=\#(\term_Z\cap \term_{W'})=2$. 
\item[(ii)] there are $W,W'$ in $\T^3_{i,j}$ and $W''$ in $\T^2_{i,j}$ such that $W\subseteq Z\subseteq W'\wed W''$ and   $\#(\term_Z\cap \term_W)=2$,  $\#(\term_Z\cap \term_{W'})=\#(\term_Z\cap \term_{W''})=1$.  
 \end{itemize}

If $k_Z=4$, then Lemma \ref{res4} implies that (i) and (ii) do not hold, hence $t^+_Z\le 3$ and $-k_Z/2\le \deg_Z L<k_Z/2$. If  
$k_Z\ge 5$, then $|\deg_Z L|<k_Z/2$. 
 \smallskip

 \emph{Case 2}.  Suppose that $C_i\subseteq Z$ and $C_j\cup C_1\subseteq Z^c$. 
  If  $W$ is a $Z$-terminal tail in $\T^2_{i,j}\cup\T^3_{i,j}$, then $Z\subseteq W$, because
   $Z$ is $\T_{i,j}$-normalized. In particular, we have $t^-_Z=0$, and $\deg_Z L=-1+t^+_Z$. 
 
 If $k_Z$ is $2$ or $3$, then notice that $t^+_Z\le k_Z-1$. Indeed,  if $t^+_Z=k_Z$, then it would follow from Lemma \ref{res3} that there is a tail $W$ in $\T^2_{i,j}\cup\T^3_{i,j}$ such that 
 $Z\subseteq W$ and $\term_Z\subseteq \term_W$. Since $Z$ and $W$ are tails, we would get $\term_Z=\term_W$, and hence $Z=W$, which is a contradiction because $C_j\subseteq \ol{W\setminus Z}$. 
 Thus, if $k_Z$ is $2$ or $3$, then  $-1\le \deg_Z L\le -2+k_Z$, and hence $-k_Z/2\le \deg_Z L<k_Z/2$.  
 
 If $k_Z\ge 4$, then Lemma \ref{res2} implies that $t^+_Z\le 2$, and hence $|\deg_Z L|\le 1<k_Z/2$.
 
 \smallskip
 \emph{Case 3}. Suppose that $C_i\cup C_j\cup C_1\subseteq Z^c$. 
 If $W$ is a $Z$-terminal tail in $\T^2_{i,j}\cup\T^3_{i,j}$, then either $Z\subseteq W$ or $W\subseteq Z^c$, because $Z$ is $\T_{i,j}$-normalized, and hence $\deg_Z L=t^+_Z-t^-_Z$.

Assume  $k_Z=2$. We claim that $t^-_Z\le 1$. Indeed, if $t^-_Z=2$, we have two cases.  
 In the first case, there are tails $W$ in $\T^2_{i,j}$ and $W'$ in $\T^3_{i,j}$, such that $Z,W^c,(W')^c$ satisfy the hypotheses of Lemma 
 \ref{res3}, which is a contradiction. In the second case, there is a tail $W$ in $\T^2_{i,j}\cup \T^3_{i,j}$ such that $\term_Z\subseteq\term_W$ and $W\subseteq Z^c$, which implies that $W$ is equal to $Z^c$,  again a contradiction, because $C_1$ is contained in $\ol{Z^c\setminus W}$.  
 
We also claim that $t^+_Z\le 1$, and if the equality hold, then $t^-_Z=1$. Indeed, assume that   
$t^+_Z\ge 1$. We have two cases.  
 In the first case,  there is $W$ in $\T^2_{i,j}\cup\T^3_{i,j}$ such that $\#(\term_W\cap \term_Z)=1$ and $Z\subseteq W$. In this case, it follows from  Lemma \ref{res5} that $t^-_Z\ne 0$, and hence 
 $t^+_Z=t^-_Z=1$. In the second case, there is $W$ in $\T^2_{i,j}\cup\T^3_{i,j}$ such that $Z\subseteq W$ and $\term_Z\subseteq \term_W$, which implies that $Z$ is equal to $W$, again a contradiction because $C_i$ is contained in $\ol{W\setminus Z}$.
In this way,  if $k_Z=2$, then $\deg_Z L=t^+_Z-t^-_Z$, and it is $-1$ or $0$.

Assume  $k_Z=3$. We claim that $t^-_Z\le 2$, and if the equality holds, then $t^+_Z=1$. Indeed, if $t^-_Z=3$, then applying Lemma \ref{res1} to $Z^c$, we see that there is a tail $W$ in $\T^3_{i,j}$ such that 
$\term_Z=\term_W$ and $W\subseteq Z^c$, and hence $W$ is equal to $Z^c$, which is a contradiction because $C_1$ is contained in $\ol{Z^c\setminus W}$. Therefore, we have $t^-_Z\le 2$. Furthermore,  if $t^-_Z=2$, then using again Lemma \ref{res1} we get a tail 
$W$ in $\T^3_{i,j}$ such that $W\subseteq Z^c$ and $\#(\term_Z\cap \term_W)=2$. In this way, Lemma \ref{res6} implies the existence of a $Z$-terminal tail $W'$ in $\T^2_{i,j}$ such that $Z\subseteq W'$, and hence $t^+_Z=1$.

We also claim that $t^+_Z\le 2$, and if the equality holds, then $t^-_Z=1$. Indeed, if $t^+_Z=3$, then it follows from  Lemma \ref{res3} that there is $W$ in $\T^3_{i,j}$ such that $\term_Z=\term_W$ and $Z\subseteq W$, and hence $Z$ is equal to $W$, which is a contradiction because $C_i$ is contained in $\ol{W\setminus Z}$. 
 Furthermore, if $t^+_Z=2$, then we have two cases. In the first case, there is a tail 
$W$ in $\T^3_{i,j}$ such that $\#(\term_Z\cap \term_W)=2$ and $Z\subseteq W$, and hence the hypotheses of Lemma \ref{res5} are satisfied. In the second case, there are tails $W$ in $\T^2_{i,j}$ and $W'$ in $\T^3_{i,j}$  such that $\#(\term_Z\cap \term_W)=\#(\term_Z\cap \term_{W'})=1$ and $Z\subseteq W\wed W'$, and hence the hypotheses of Lemma \ref{res7} are satisfied. In any case, the cited Lemmas imply the existence of a $Z$-terminal tail $W''$ in $\T^2_{i,j}\cup\T^3_{i,j}$ such that $W''\subseteq Z^c$,  hence $t^-_Z=1$. Thus, if $k_Z=3$, then $|\deg_Z L|=|t^+_Z-t^-_Z|\le 1<k_Z/2$.

If $k_Z\ge 4$, it follows from  Lemmas \ref{res1} and \ref{res2} that $t^+_Z\le 2$ and $t^-_Z\le 2$, and hence $|\deg_Z L|=|t^+_Z-t^-_Z|\le 2$. To show that the condition $-k_Z/2\le \deg_Z L < k_Z/2$ holds, we only need to rule out the case 
$\deg_Z L=k_Z/2=2$. But, if $\deg_Z L=k_Z/2=2$, then $k_Z=4$, $t^+_Z=2$ and $t^-_Z=0$, and using Lemmas \ref{res5} and \ref{res7} as we did in the case $k_Z=3$, we would get $t^-_Z\ge 1$, which is a contradiction.
 
 In the remaining cases where $C_1$ is contained in $Z$, we can apply what we proved so far to $Z^c$, and we obtain  $-k_Z/2<\deg_Z L\le k_Z/2$. 
  \end{proof}
 
 \begin{Exa}\label{Exa}
 Let $\C\ra B$ be a regular local smoothing of a nodal curve $C$ as in the figure,  and fix smooth points $P$, $Q$ and $Q'$ of $C$ such that $P\in C_1$ and $\{Q, Q'\}\subset C_4$.
 
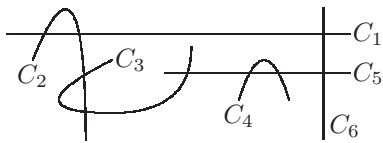
\[
\begin{xy} <10pt,0pt>:
(0,0)*{\scriptstyle}="a"; 
"a"+(-9,0);"a"+(4,0)**\dir{-}; 
"a"+(3,1);"a"+(3,-4)**\dir{-}; 
"a"+(4,-1.5);"a"+(-3,-1.5)**\dir{-}; 
"a"+(-8,-1);"a"+(-6,-4)**\crv{(-6,4)}; 
"a"+(-5,-1);"a"+(-5,-3)**\crv{(-9,-3)}; 
"a"+(-2,-0.5);"a"+(-5,-3)**\crv{(-2,-3)}; 
"a"+(-0.2,-2.5);"a"+(1.7,-2.5)**\crv{(0.75,0.5)}; 
"a"+(4.7,0)*{C_1};
"a"+(4.7,-1.5)*{C_5};
"a"+(-8,-1.6)*{C_2};
"a"+(-4.3,-1)*{C_3};
"a"+(-0.2,-3.1)*{C_4};
"a"+(3.8,-3.5)*{C_6};
\end{xy}
\]
The curve $C$ has no 1-tails. The set of nested 2-tails (respectively 3-tails) of $C$ with respect to $(4, 4)$ is $\T^2_{4,4}=\{C_4, C_4\cup C_5\}$ (respectively $\T^3_{4, 4}=\{C_3\cup C_4\cup C_5\cup C_6\}$). Set $L:=\O_C(2P-Q-Q')\otimes \O_\C(-C_3-3C_4-2C_5-C_6)|_C$.  We have 
$$\deg_{C_1} L=\deg_{C_3}L=1, \,\, \deg_{C_2} L=-2, \,\, \deg_{C_4}L=\deg_{C_5}L=\deg_{C_6}=0,$$ 
  and hence $L$ is $C_1$-quasistable, as predicted by Theorem \ref{main1}.
 \end{Exa}


\section*{Acknowledgments}
\noindent
This paper is a continuation of a project started in collaboration with J. Coelho and E. Esteves.
I wish to thank them for several stimulating discussions.

\end{document}